\documentclass[a4paper]{amsart}

\usepackage[utf8]{inputenc}
\usepackage[hyphenbreaks]{breakurl}
\usepackage[hyphens]{url}
\usepackage[colorlinks, final]{hyperref}
\usepackage{a4wide}
\usepackage{amsmath, amscd, amsfonts, amsthm, amssymb, bbm, mathtools}
\usepackage{apptools}
\usepackage{aliascnt}
\usepackage{comment}
\usepackage{float, graphicx, tikz}
\usepackage{xkvltxp, fixme}

\fxsetup{targetlayout=color}

\makeatletter

\newcommand{\@pmath}{$p$}

\title[The Number of Roots of a Random Polynomial over $\QQ_p$]{The Number of Roots of a Random Polynomial over The Field of \texorpdfstring{\@pmath}{p}-adic Numbers}
\author{Roy Shmueli}

\address{Raymond and Beverly Sackler School of Mathematical Sciences, Tel Aviv University, Tel Aviv 69978, Israel}
\email{royshmueli@tauex.tau.ac.il}

\newcommand{\autotheorem}[3]{
\newaliascnt{#1counter}{#2}
\newtheorem{#1}[#1counter]{#3}
\expandafter\newcommand\csname #1counterautorefname\endcsname{#3}
}

\theoremstyle{plain}

\autotheorem{thm}{theorem}{Theorem}
\autotheorem{lem}{theorem}{Lemma}
\autotheorem{prop}{theorem}{Proposition}
\autotheorem{cor}{theorem}{Corollary}
\autotheorem{fact}{theorem}{Fact}

\theoremstyle{definition}
\autotheorem{defn}{theorem}{Definition}
\autotheorem{asmp}{theorem}{Assumption}
\autotheorem{exm}{theorem}{Example}

\theoremstyle{remark}
\autotheorem{rem}{theorem}{Remark}

\AtAppendix{\counterwithin{theorem}{section}}
\AtAppendix{\counterwithin{equation}{section}}

\hypersetup{
  pdftitle={\@title},
  pdfauthor={\@author}
}

\newcommand*{\ZZ}{{\mathbb{Z}}}
\newcommand*{\QQ}{{\mathbb{Q}}}
\newcommand*{\RR}{{\mathbb{R}}}
\newcommand*{\CC}{{\mathbb{C}}}
\newcommand*{\FF}{{\mathbb{F}}}

\newcommand{\Acl}{\mathcal{A}}
\newcommand{\Bcl}{\mathcal{B}}

\newcommand{\Scl}{\mathcal{S}}

\renewcommand*{\Pr}{{\mathbb{P}}}
\newcommand*{\Ex}{{\mathbb{E}}}
\newcommand*{\cond}{\;\middle|\;}

\DeclarePairedDelimiter{\pa}{\lparen}{\rparen}
\DeclarePairedDelimiter{\br}{\lbrack}{\rbrack}
\DeclarePairedDelimiter{\ang}{\langle}{\rangle}
\DeclarePairedDelimiter{\set}{\{}{\}}
\DeclarePairedDelimiter{\abs}{\lvert}{\rvert}
\DeclarePairedDelimiter{\floor}{\lfloor}{\rfloor}

\newcommand*{\rv}{{\xi}}
\newcommand*{\rvs}{{\eta}}
\newcommand*{\dlt}{{\varepsilon_1}}

\renewcommand*{\v}{{\vec{v}}}
\renewcommand*{\u}{{\vec{u}}}
\newcommand*{\w}{{\vec{w}}}
\newcommand*{\x}{{\vec{x}}}
\newcommand*{\y}{{\vec{y}}}
\newcommand*{\vze}{{\vec{0}}}

\newcommand*{\@rdsetExplicit}[2][d]{{R^{\pa{#1}}\pa*{#2}}}
\newcommand*{\@rdsetOne}[1]{{R\pa*{#1}}}
\newcommand{\rdset}{\@ifstar{\@rdsetOne}{\@rdsetExplicit}}
\newcommand{\rfield}[2][\QQ_p]{R_{#1}\pa*{#2}}
\newcommand*{\hdset}[3][d]{{H^{\pa{#1}}_{#2}\pa*{#3}}}
\newcommand*{\rdsetu}[2][d]{{\tilde{R}^{\pa{#1}}\pa*{#2}}}

\newcommand*{\stirling}[2]{\genfrac{\{}{\}}{0pt}{}{#1}{#2}}

\newcommand*{\rmod}[1]{{\ZZ/{#1}\ZZ}}

\newcommand*{\eqsys}[2]{{\left\{\begin{array}{l} {#1} \\ {#2} \end{array}\right.}}

\newcommand*{\tindex}{{\substack{0\le j < m \\ 1 \le r < p}}}
\newcommand*{\tindexarg}[2][p]{{\substack{0\le j < {#2} \\ 1 \le r < #1}}}

\newcommand*{\gseq}{\vec{g}}
\newcommand*{\hseq}{\vec{h}}
\newcommand*{\polysplit}[1][f]{\hat{#1}}

\newcommand*{\indic}[1]{\mathbbm{1}_{#1}}

\newcommand*{\lift}[1]{L\pa*{#1}}

\DeclareMathOperator{\he}{H}
\DeclareMathOperator{\weight}{Weight}
\DeclareMathOperator{\Var}{Var}
\DeclareMathOperator{\Res}{Res}

\renewcommand{\pod}[1]{\allowbreak\mathchoice
  {\if@display \mkern 18mu\else \mkern 8mu\fi (#1)}
  {\if@display \mkern 18mu\else \mkern 8mu\fi (#1)}
  {\mkern4mu(#1)}
  {\mkern4mu(#1)}
}

\makeatother

\begin{document}

\begin{abstract}
We study the roots of a random polynomial over the field of p-adic numbers.
For a random monic polynomial with coefficients in $\ZZ_p$, we obtain an asymptotic formula for the
factorial moments of the number of roots of this polynomial.
In addition, we show the probability that a random polynomial of degree $n$ has more than $\log n$ roots is $O\pa*{n^{-K}}$ for some $K > 0$.
\end{abstract}

\maketitle

\listoffixmes

\section{Intorduction}
\label{sec:introduction}

Consider the random polynomial
\begin{equation*}
  f\pa*{X} = \rv_0 + \rv_1 X + \dots + \rv_n X^n
\end{equation*}
where $\rv_0, \dots, \rv_n$ are independent random variables taking values in some field $F$.
For a subset, $E$, of the algebraic closure of $F$, we define $\rfield[E]{f}$ to be the number of distinct roots in $E$ of the polynomial $f$, i.e.
\begin{equation*}
\rfield[E]{f} = \#\set*{x \in E : f\pa*{x} = 0}.
\end{equation*}

The study of the distribution of $\rfield[\RR]{f}$ when $F = \RR$, has a long history.
It goes back to Bloch and P\'{o}lya \cite{bloch1932roots} who showed that $\Ex\br*{\rfield[\RR]{f}} = O\pa*{\sqrt{n}}$ as $n\to\infty$ when the coefficients of $f$ takes the values of $\pm1$ both with equal probability.
Later, their results were improved and generalized on many occasions, see \cite{littlewood1938number, littlewood1939number, kac1943average, erdos1956number, ibragimov1971expected, nguyen2016number, soze2017real1, soze2017real2}.
In particular, Maslova~\cite{maslova1974distribution, maslova1974variance} determined asymptotically all higher moments of $\rfield[\RR]{f}$ as $n \to \infty$ when $f$ is a general random real polynomial.

Evans~\cite{evans2006expected} studied $\rfield[\QQ_p]{f}$ in the $p$-adic setting, i.e., $E = F =\QQ_p$ and in fact a multi-variate version.
In his model, the coefficients are not independent and the randomicy comes from the Haar measure (cf.\ \cite[Theorem 5]{kulkarni2021padic} for generalizations).
Buhler, Goldstein, Moews, and Rosenberg~\cite{buhler2006probability} worked in our setting where th coefficient are independent and distributed according to the Haar measure on $\ZZ_p$. They computed the probability that $\rfield[\QQ_p]{f} = n$, i.e.\ $f$ totally splits, conditioned on $
\rv_n=1$. Recently, Caruso \cite{caruso2021zeroes} generalized their result and obtained an integral formula for $\rfield[E]{f}$, for any open subset $E$ of finite field extension $K/\QQ_p$.

In \cite{shmueli2021expected}, the author estimated $\Ex\br*{\rfield{f}}$ for general distributions of $f$.
For example, if $p\ne 2$ and $\rv_0, \dots, \rv_{n-1}$ are taking values of $\pm1$ both with equal distribution then for any $\varepsilon > 0$
\begin{equation*}
\Ex\br*{\rfield{f}} = \frac{p-1}{p+1} + O\pa*{n^{-1/4 + \varepsilon}}
\end{equation*}
as $n\to\infty$.

In this paper, we will consider the case where $f$ is a monic polynomial with $p$-adic integer coefficients, i.e., $\rv_n = 1$ and $\rv_0, \dots, \rv_{n-1} \in \ZZ_p$.
In this case, we have that $\rfield{f} = \rfield[\ZZ_p]{f}$, so we abbreviate and write $\rdset*{f} = \rfield[\ZZ_p]{f}$.

Our study is focused on the factorial moments of $\rdset*{f}$ which is defined as follows.
We call a set of exactly $d$ elements a $d$-set, and we denote $\rdset{f}$ to be the number of $d$-sets of roots of $f$ i.e.
\begin{equation*}
\rdset{f} = \binom{\rdset*{f}}{d}.
\end{equation*}
The expected value of $\rdset{f}$ is called the $d$-th factorial moment of $\rdset*{f}$.
Computing the factorial moments of $\rdset*{f}$ is equivalent to computing its distribution and its standard moments, using:
\begin{equation}
\label{eq:factorial-to-standard-moment}
\Ex\br*{\rdset*{f}^m} = \sum_{d=1}^m \stirling{m}{d} d!\, \Ex\br*{\rdset{f}},
\end{equation}
where the curly braces denote Stirling numbers of the second kind.

Bhargava, Cremona, Fisher and Gajovi\'c in \cite{bhargava2021density}, computed the $d$-factorial moments $\rdset*{f}$ when $\rv_i$ is distributed according to Haar measure normalized on $\ZZ_p$ and $p\ZZ_p$.
They showed that for $n \ge 2d$ the expectation $\Ex\br*{\rdset{f}}$ is independent of $n$.
Moreover, if we denote $\alpha\pa*{d}$ (respectively $\beta\pa*{d}$) to be $\Ex\br*{\rdset{f}}$ when $n \ge 2d$ and $\rv_i$ is distributed according to Haar measure normalized on $\ZZ_p$ (respectively $p\ZZ_p$), then we have the following power series relation between them:
\begin{equation}
\label{eq:alpha-beta-power-series}
  \sum_{d=0}^\infty \alpha\pa*{d} t^d = \pa*{\sum_{d=0}^\infty \beta\pa*{d} t^d}^p.
\end{equation}

Our result deals with a rather general distribution for the coefficients.
Using a similar power series, we define $\gamma\pa*{d}$ by
\begin{equation}
\label{eq:gamma-def}
\sum_{d=0}^\infty \gamma\pa*{d} t^d = \pa*{\sum_{d=0}^\infty \beta\pa*{d} t^d}^{p-1}.
\end{equation}
The values of $\gamma\pa*{d}$ satisfy the following theorem.
\begin{thm}
\label{thm:main-small}
Let $f\pa*{X}=\rv_{0}+\rv_{1}X+\dots+\rv_{n-1}X^{n-1}+X^{n}$ where $\rv_{0},\dots,\rv_{n-1}$ are i.i.d.\ random variables taking values in $\ZZ_{p}$ such that $\rv_i\bmod p$ is non-constant for all $i=0,\dots, n-1$.
Then for any $d = d\pa*{n} = o\pa*{\log^{1/2} n}$ and any $\varepsilon > 0$ we have
\begin{equation*}
  \Ex\br*{\rdset{f} \cond p\nmid \rv_0} =
    \gamma\pa*{d} + O\pa*{n^{-1/4+\varepsilon}}
\end{equation*}
as $n\to\infty$. Here the implied constant depends only on $p$, $\varepsilon$ and the distribution of $\rv_0$.
\end{thm}

For larger values of $d$, the main term of \autoref{thm:main} is smaller than the error term since $\gamma\pa*{d} \le e^{-c d^2}$ for some $c > 0$, see \autoref{gamma-estimate}.
For larger $d$, we take a wider point of view.
\begin{asmp}
\label{main-assumption}
Assume $\rv_0, \dots, \rv_{n-1}$ are independent random variables taking values in $\ZZ_p$ such that there exists $0 < \tau < 1$ independent of $n$ which satisfies that
\begin{equation*}
  \sum_{\bar{x}\in\rmod p} \Pr\pa*{\rv_i \equiv \bar{x} \pmod{p}}^{2} < 1 - \tau
\end{equation*}
for each $i=1,\dots,n-1$.
\end{asmp}
Note that \autoref{main-assumption} does not give any requirements on $\rv_0$ beside being independent of the other $\rv_i$ and taking values in $\ZZ_p$.
\begin{thm}
\label{thm:main}
Let $f\pa*{X}=\rv_{0}+\rv_{1}X+\dots+\rv_{n-1}X^{n-1}+X^{n}$ where $\rv_0,\dots,\rv_{n-1}$ satisfy \autoref{main-assumption}.
Then for any $d = d\pa*{n}$ such that $\limsup_{n\to\infty} d / \log n < \pa*{16 \log p}^{-1}$ there exists an explicit constant $C > 0$ such that for any $\varepsilon > 0$
\begin{equation*}
  \Ex\br*{\rdset{f} \cond p\nmid \rv_0} =
    \gamma\pa*{d} + O\pa*{n^{-C+\varepsilon}}
\end{equation*}
as $n\to\infty$.
Here the implied constant depends only on $p$, $\varepsilon$ and $\tau$ from \autoref{main-assumption}.
\end{thm}
The constant $C$ is determined by the following
\begin{equation}
\label{eq:main-exponent-const}
C = \frac{1}{4} -\frac{1}{4}  \he_p \pa*{4\log p\cdot \limsup_{n\to \infty} \frac{d}{\log n}},
\end{equation}
where $\he_p$ is the binary entropy function with base $p$, defined by
\[
\he_p \pa*{x} = x \log_p \frac{1}{x} + \pa*{1-x} \log_p \frac{1}{1-x}.
\]
If the random variables $\rv_0, \dots, \rv_{n-1}$ are i.i.d.\ and $\rv_i\bmod{p}$ is non-constant, then they satisfy \autoref{main-assumption}.
Thus, \autoref{thm:main-small} follows from \autoref{thm:main} and \eqref{eq:main-exponent-const}.

Some special cases that are studied frequently when the coefficients are i.i.d.\ distributed uniformly in one of the sets $\set*{-1,0,1}$, $\set*{0,1}$ or $\set*{0,\dots, p-1}$.
In those cases, we can estimate $\Ex\br*{\rdset{f}}$ without conditioning on $\rv_0$ using the following:
\begin{thm}
\label{main-with-nonunit-root}
Let $f\pa*{X}=\rv_{0}+\rv_{1}X+\dots+\rv_{n-1}X^{n-1}+X^{n}$ be a random polynomial where $\rv_{0},\dots,\rv_{n-1}$ are random variables taking values in $\ZZ_{p}^\times\cup\set*{0}$ and satisfying \autoref{main-assumption}.
Then for any $d = d\pa*{n}$ such that $\limsup_{n\to\infty} d / \log n < \pa*{16 \log p}^{-1}$ there exists an explicit constant $C > 0$, defined in \eqref{eq:main-exponent-const}, such that for any $\varepsilon > 0$
\begin{equation*}
  \Ex\br*{\rdset{f}} =
    \gamma\pa*{d} + \Pr\pa*{\rv_0 = 0} \gamma\pa*{d-1} + O\pa*{n^{-C+\varepsilon}}
\end{equation*}
as $n\to\infty$.
Here the implied constant depends only on $p$ and $\varepsilon$.
\end{thm}

\autoref{thm:main} combined with \eqref{eq:factorial-to-standard-moment} give us a way to estimate any fixed moment of $\rdset*{f}$.
In particular, we can compute the expected value and the variance of $\rdset*{f}$ as follows:
\begin{cor}
\label{ex-and-var}
Let $f\pa*{X}=\rv_{0}+\rv_{1}X+\dots+\rv_{n-1}X^{n-1}+X^{n}$ where $\rv_0,\dots,\rv_{n-1}$ satisfy \autoref{main-assumption} then for any $\varepsilon > 0$
\begin{align*}
\Ex\br*{\rdset*{f} \cond p\nmid \rv_0} &= \frac{p-1}{p+1} + O\pa*{n^{-1/4+\varepsilon}} \qquad\text{and} \\
\Var\br*{\rdset*{f} \cond p\nmid \rv_0} &= \frac{\pa*{p^2+1}^2 \pa*{p-1}}{\pa*{p^4+p^3+p^2+p+1} \pa*{p+1}} + O\pa*{n^{-1/4+\varepsilon}}
\end{align*}
as $n\to\infty$.
Here the implied constant depends only on $p$, $\varepsilon$ and $\tau$ from \autoref{main-assumption}.
\end{cor}

We can also use the results with Markov's inequality to obtain a bound on the probability that a random polynomial has a large number of roots.
\begin{cor}
\label{large-root-count-bound}
Let $f\pa*{X}=\rv_{0}+\rv_{1}X+\dots+\rv_{n-1}X^{n-1}+X^{n}$ where $\rv_0,\dots,\rv_{n-1}$ satisfy \autoref{main-assumption}.
Then
\begin{enumerate}
\item \label{lrcb:log-case} There exists a constant $K > 0$ such that
\begin{equation*}
\Pr\pa*{\rdset*{f} \ge \log n \cond p \nmid \rv_0} = O\pa*{n^{-K}}
\end{equation*}
as $n\to\infty$.

\item \label{lrcb:power-case} For any $0 < \lambda \le 1$ there exists a constant $K > 0$ such that
\begin{equation*}
\Pr\pa*{\rdset*{f}when \ge n^\lambda \cond p \nmid \rv_0} = O\pa*{\exp\pa*{-K \log^2 n}}
\end{equation*}
as $n\to\infty$.
In particular, if $\lambda = 1$ then
\begin{equation*}
\Pr\pa*{f\text{ totally split} \cond p \nmid \rv_0} = O\pa*{\exp\pa*{-K \log^2 n}}
\end{equation*}
as $n\to\infty$.
\end{enumerate}
\end{cor}

\subsection{Structure of the paper}
\autoref{sec:preliminaries} contains generalizations and variants of known facts regarding $p$-adic number, uniform random $p$-adic polynomials, and random walks.
In \autoref{sec:ups-space}, we study the distribution of $\rdset*{g}$ for specific family of polynomials $g$ and in \autoref{sec:hasse} we study the distribution of the Hasse derivatives of $f$.
We prove \autoref{thm:main} and \autoref{main-with-nonunit-root} in \autoref{sec:main}, mainly using \autoref{ex-sum-prod} which is described and proved in the same section.
Finally, in \autoref{sec:large-root-count} we bound the probability of $f$ having a large number of roots, see \autoref{large-root-count-bound}.

\subsection{Basic Notations and conventions}
In this paper, we will assume that $p$ is a fixed prime and all the implied constants of the big O notation may depend on $p$.

For a ring $A$, we denote the ring of polynomials over $A$ with $A\br*{X}$.
For $n \ge 0$ we denote $A\br*{X}_n$ to be the subset of $A\br*{X}$ contains all polynomials of degree $n$ and denote $A\br*{X}_n^1$ to be the subset of $A\br*{X}$ contains all monic polynomials of degree $n$.


\subsection*{Acknowledgements}
I would like to thank my supervisor,  Lior Bary-Soroker, for his guidance.
I also thank Eli Glasner for his support in the research and Itai Bar-Deroma for many helpful conversations.

This research was partially supported by a grant from the Israel Science Foundation, grant no.\ 1194/19.


\section{Preliminaries}
\label{sec:preliminaries}

\subsection{The \texorpdfstring{$p$}{p}-adic numbers}
For a fixed prime number $p$, we can write any non-zero rational number $x\in\QQ^{\times}$
as $x=p^{t}\cdot\frac{a}{b}$ such that $a, b, t \in \ZZ$ and $p\nmid a,b$.
We use this factorization to define \emph{the $p$-adic absolute value}:
\begin{equation*}
\abs*{x}_p= \begin{cases}
p^{-t}, &x\ne 0 \\
0, &x = 0.
\end{cases}
\end{equation*}
The absolute value $\abs{\,\cdot\,}_p$ satisfies:
\begin{equation*}
\begin{split}
	&\abs*{x}_p \ge 0 \quad\text{and}\quad \abs{x}_{p}=0\iff x=0\text{,} \\
	&\abs*{xy}_p  =\abs*{x}_p \abs*{y}_p\text{,}\\
	&\abs*{x+y}_p \le \max\pa*{\abs*{x}_p, \abs*{y}_p}\text{.}
\end{split}
\end{equation*}

We define \emph{the field of $p$-adic numbers}, denoted by $\QQ_p$, as the completion of $\QQ$ with respect to $\abs{\,\cdot\,}_p$. We define \emph{the ring of $p$-adic
integers}, denoted by $\ZZ_p$, as the topological closure of $\ZZ$ in $\QQ_p$.
Then,
\[
  x \in \ZZ_p \iff \abs*{x}_p \le 1\text{.}
\]

The ring $\ZZ_p$ is local with  maximal
ideal $p\ZZ_{p}$. All the non-zero ideals are of the form
$p^{k}\ZZ_{p}$ for some integer $k\geq 0$. The quotient ring
$\ZZ_{p}/p^{k}\ZZ_{p}$ is canonically isomorphic to the ring $\ZZ/p^{k}\ZZ$.
Therefore we use the
notation of reduction modulo $p^{k}$ as in the integers, i.e., for $x, y\in \ZZ_p$ we write
\[
x\equiv y \pmod{p^k} \iff x-y\in p^k\ZZ_p\text{.}
\]
Note that $x\equiv y \pmod{p^k} \iff \abs*{x-y}_p \le p^{-k}$
and that $x=0\iff x\equiv0\pmod{p^k}$ for all $k\ge 1$.

Our proof utilizes the following generalizations of Hensel's lemma.
The first generalization is also called Newton-Raphson method in $p$-adic fields see \cite[Theorem~II.4.2]{bachman1964introduction}, \cite[Proposition~II.2]{lang1970algebraic}, \cite[Theorem~7.3]{eisenbud2013commutative} or \cite[Theorem~4.1]{conradXXXXhensel} for slightly weaker versions. For exact proof see \cite[Theorem 4]{shmueli2021expected}.
\begin{thm}[Newton-Raphson method]
\label{pre:newton-raphson}
If $f\in\ZZ_{p}\br*{X}$
and $\bar{x}\in\rmod{p^{2k-1}}$ satisfies
\[
f\pa*{\bar x} \equiv0 \pmod{p^{2k-1}} \quad\text{and}\quad
f'\pa*{\bar x} \not \equiv0 \pmod{p^k},
\]
then $\bar x$ can be lifted uniquely from $\rmod{p^k}$ to a root of
$f$ in $\ZZ_p$, i.e., there is a unique $x\in\ZZ_{p}$ such
that $f\pa*{x}=0$ and $x\equiv \bar x\pmod{p^k}$.
\end{thm}

The other generalization is used to factor polynomial in $p$-adics fields, see
\cite[Proposition 3.5.2]{fried2005field} or
\cite[Lemma II.4.6]{neukirch1999algebraic}.
\begin{thm}[Hensel's Lemma]
\label{pre:hensels-lemma}
Let $f\in\ZZ_p\br*{X}$ be a polynomial
and $\bar{g}, \bar{h} \in \rmod{p}\br*{X}$ coprime polynomials satisfying $f \equiv \bar{g}\bar{h} \pmod{p}$.
Then $\bar{g}, \bar{h}$ can be lifted uniquely to polynomials $g, h \in \ZZ_p\br*{X}$ such that $f=gh$.
Moreover, if $f$ is monic polynomial then also $g$ and $h$ are also monic polynomials.
\end{thm}

We also note another lemma regarding random polynomials over $\QQ_p$.
\begin{lem}
\label{pre:nonunit-roots}
Let $f\pa*{X}=\rv_{0}+\rv_{1}X+\dots+\rv_{n-1}X^{n-1}+X^{n}$ be a random polynomial where $\rv_{0},\dots ,\rv_{n-1}$ are random variables taking values in $\ZZ_{p}^\times\cup\set*{0}$.
Let $f_0\pa*{X} = f\pa*{pX}$, then $f_0$ has no non-zero roots in $\ZZ_p$ almost surely.
\end{lem}
\begin{proof}
The proof of \cite[Lemma 20]{shmueli2021expected} gives the stronger statement of our lemma.
\end{proof}

\subsection{Unifrom random \texorpdfstring{$p$}{p}-adic polynomial}
The $p$-adic absolute value induces a metric on $\QQ_p$ defined by $d\pa*{x, y} = \abs*{x-y}_p$.
The open balls of this metric are of the form $x+p^k\ZZ_p$
for some $x\in\QQ_p$ and $k\in\ZZ$.
Since the
$p$-adic absolute value is discrete, every open ball is also closed
and compact.
By Haar's theorem
(see \cite[Chapter XI]{halmos1974measure}), there exists up to a positive multiplicative constant, a unique regular non-trivial measure $\mu$ on Borel subsets of $\QQ_{p}$ such that for any Borel set
$E\subseteq\QQ_{p}$ and $x\in\QQ_{p}$,
\begin{equation*}
\mu\pa*{x+E} = \mu\pa*{E} \qquad\text{and}\qquad
\mu\pa*{xE} = \abs{x}_{p}\mu\pa*{E}.
\end{equation*}
This measure is called a \emph{Haar measure on $\QQ_p$}.

For a compact set $K \subseteq \QQ_p$, we call $\mu$ \emph{the Haar measure normilized on $K$} if $\mu$ is a Haar measure on $\QQ_p$ and $\mu\pa*{K} = 1$.
In the vector space $\QQ_p^n$, for a compact set $K \subseteq \QQ_p^n$ we say  $\mu$ is \emph{the Haar measure normalized on $K$} if it is a product of some Haar measures on $\QQ_p$ and $\mu\pa*{K} = 1$.
The Haar measure normalized on $K$ is unique and always exists.

We use the embedding of $\QQ_p\br*{X}^1_n$ in $\QQ_p^n$:
\[
a_0 + a_1 X + \dots + a_{n-1} X^{n-1} + X^n \mapsto \pa*{a_0, \dots, a_{n-1}},
\]
to define a topology in $\QQ_p\br*{X}_n^1$ and equivalent definition of Haar measure on normalized compact subsets.
Moreover, for a compact set $P \in \QQ_p\br*{X}^1_n$ we say that a random polynomial $h \in \QQ_p\br*{X}^1_n$ is \emph{distributed uniformly on $P$} if $h$ is taking values in $P$ and distributed according to the Haar measure normalized on $P$ and restricted to $P$.

For integers $m \le n$ we define the set $P_{m, n} \subseteq \ZZ_p\br*{X}_n^1$ to be the set of all monic polynomials $h\pa*{X} = a_0 + a_1 X + \dots + a_{n-1} X^{n-1} + X^n$ such that $a_m$ is the first coefficient not divisible by $p$.
We note two special cases of this set when $m=n$ and when $m=0$.
The set $P_{n,n}$ is the set of all monic polynomials $f\in \ZZ_p\br*{X}_n^1$ which their reduction modulo $p$ is $X^n$, i.e.\
\[
P_{n,n} = \set*{f\in \ZZ_p\br*{X}_n^1 : f \equiv X^n \pmod{p}}.
\]
And $P_{0,n}$ is the set of all monic polynomials $g\in \ZZ_p\br*{X}_n^1$ such that the free coefficient $g\pa*{0}$ is not divisible by $p$, i.e.\
\[
P_{0,n} = \set*{g\in \ZZ_p\br*{X}_n^1 : g\pa*{0} \not\equiv 0 \pmod{p}}.
\]

\begin{lem}
\label{pre::uniform-distribution-in-P}
For any integers $m < n$, a random polynomial $h$ is distributed uniformly in $P_{m,n}$ if and only if there exists random independent polynomials $f$ and $g$ distributed uniformly in $P_{m,m}$ and $P_{0,n-m}$, respectively, such that $h=fg$.
\end{lem}
\begin{proof}
Let $f \in P_{m,m}$ and $g \in P_{0,n-m}$.
From the definition $f$ and $g$ are monic and their reductions modulo $p$ are coprime.
Therefore, we use \cite[Corollary 2.5]{bhargava2021density} to infer that the  resultant $\Res\pa*{f,g}$ is a unit in $\ZZ_p$ for all $f \in P_{m,m}$ and $g \in P_{n,m}$.
From \cite[Corollary 2.7]{bhargava2021density}, we conclude that the multiplication map $P_{m,m} \times P_{0, n-m} \to P_{m,m} P_{0, n-m}$ is measure preserving.
Thus, it is suffice to show that $P_{m,m}P_{0,n-m} = P_{n,m}$.

Let $h\pa*{X} = a_0 + a_1 X + \dots + a_{n-1} X^{n-1} + X^n \in P_{m,n}$, then taking reduction modulo $p$ gives
\begin{equation*}
h\pa*{X} \equiv X^m \pa*{a_m + \dots + a_{n-1} X^{n-m-1} + X^{n-m}} \pmod{p}.
\end{equation*}
Since $a_m \not\equiv 0 \pmod{p}$, the polynomials $X^m$ and $a_m + \dots + a_{n-1} X^{n-m-1} + X^{n-m}$ are coprime modulo $p$.
By Hensel's Lemma (\autoref{pre:hensels-lemma}), there exists a lift $f$ of $X^m$ and a lift $g$ of  $a_m + \dots + a_{n-1} X^{n-m-1} + X^{n-m} \bmod{p}$ such that $h=fg$.
The polynomial $f$ is a lift of $X^m$ hence $f\in P_{m,m}$. Also, $g\pa*{0} \equiv \rv_m \not\equiv 0 \pmod{p}$ hence $g\in P_{0, n-m}$.
So we got that $h = fg \in P_{m,m}P_{0,n-m}$.

Therefore, $P_{m,n} \subseteq P_{m,m}P_{0,n-m}$ and the other is direction is trivial when reducing the product modulo $p$.
\end{proof}

We set $\alpha\pa*{n, d} = \Ex\br*{\rdset{f}}$ (respectively $\beta\pa*{n,d}$) where $f$ is a random polynomial distributed uniformly on $\ZZ_p\br*{X}_n^1$ (respectively $P_{n,n}$).
We have the following lemmas regarding the values of $\alpha\pa*{n, d}$ and $\beta\pa*{n,d}$ which are taken from \cite{bhargava2021density}.

\begin{lem}
\label{pre:alpha-beta-recursion}
The values of $\alpha\pa*{n, d}$ and $\beta\pa*{n,d}$ can be computed using the following recurrence relation.
First, for all $n \ge d \ge 0$ we have
\begin{equation}
\label{eq:pre:alpha_in_beta_terms}
  \alpha\pa*{n,d} = p^{-n} \sum_{\bar{f}\in\FF_p\br*{X}_n^1} \; \sum_{d_0+\dots+d_{p-1} = d} \; \prod_{r=0}^{p-1} \beta\pa*{n_r, d_r}
\end{equation}
where the inner sum runs over all non-negative integers $d_0, \dots, d_{p-1}$ such that $d_0 + \dots + d_{p-1} = d$ and $n_r$ is the multiplicity of $r$ as a root of $\bar{f}$ over $\FF_p$, i.e.,
\begin{equation*}
  n_i = \max\set*{k \ge 0 : \pa*{X-i}^k \mid \bar{f}}.
\end{equation*}
Second, for all $n \ge d \ge 0$ we have
\begin{equation}
\label{eq:pre:beta_in_alpha_terms}
\beta\pa*{n,d} = p^{-\binom{n}{2}} \alpha\pa*{n, d} + \pa*{p-1} \sum_{d\le s < r < n} p^{-\binom{r+1}{2}} p^s \alpha\pa*{s, d}.
\end{equation}
And the initial conditions are
\begin{equation}
\label{eq:alpha-beta-initial}
\alpha\pa*{n, d} = \beta\pa*{n, d} = 0, \quad
\alpha\pa*{n, 0} = \beta\pa*{n, 0} = 1 \quad \text{and} \quad
\alpha\pa*{1, 1} = \beta\pa*{1, 1} = 1
\end{equation}
for all $0 \le n < d$.
\end{lem}
\begin{proof}
We start with proving the initial condition, \eqref{eq:alpha-beta-initial}.
Those identities are true because if $0\le n < d$ then $\rdset{f} = 0$ and $\rdset[0]{f} = 1$. The last identity of \eqref{eq:alpha-beta-initial} holds since all linear polynomial has exactly one root.

Next we prove \eqref{eq:pre:alpha_in_beta_terms}.
This equation is a consequence of \cite[eq. (30)]{bhargava2021density} after settings $N_\sigma = \sum_{\bar{f}\in\FF_p\br*{X}_n^1,\sigma\pa*{\bar f} = \sigma} 1$ and changing the order of summations.
Also note that in \cite[eq. (30)]{bhargava2021density} increasing $k$ by adding $n_r = 0$ does not change the inner sum since $\beta\pa*{0, 0} = 1$ and $\beta\pa*{0, d_r} = 0$ when $d_r > 0$.

Finally, we get \eqref{eq:pre:beta_in_alpha_terms} by plugging $\alpha\pa*{s,d} = 0$ for $s<d$ into \cite[eq. (33)]{bhargava2021density}.
\end{proof}

\begin{lem}
\label{pre:alpha-beta-independence}
The expectations $\alpha\pa*{n, d}$ and $\beta\pa*{n,d}$ are rational functions in $p$ and are independent of $n$ for $n \ge 2d$.
Moreover, we have the following equality of power series
\begin{equation*}
\sum_{d=0}^\infty \alpha\pa*{2d, d} t^d
  = \pa*{\sum_{d=0}^\infty \beta\pa*{2d, d} t^d}^p.
\end{equation*}
\end{lem}
\begin{proof}
By \cite[Theorem 1.(a) and Theorem 1.(c)]{bhargava2021density} $\alpha\pa*{n, d}$ and $\beta\pa*{n,d}$ are rational functions in $p$ and are independent of $n$ for $n \ge 2d$.
We use the notation of $\Acl_d$ and $\Bcl_d$ as defined in \cite{bhargava2021density}.
According to \cite[Theorem 1.(c)]{bhargava2021density} and the identity $\Acl_d\pa*{1} = \Acl_d\pa*{p}$ (see paragraph after \cite[eq. (38)]{bhargava2021density}) we have
\begin{equation}
\label{eq:abi:power-series-identities}
\alpha\pa*{2d, d} = \Acl_d\pa*{1} = \Acl_d\pa*{p} \quad \text{and} \quad \beta\pa*{2d, d} = \Bcl_d\pa*{1}.
\end{equation}
The equality of power series is proved by setting $t = 1$ in \cite[eq. (5)]{bhargava2021density} and then plugging~\eqref{eq:abi:power-series-identities}.
\end{proof}

Due to \autoref{pre:alpha-beta-independence} we can abbreviate and write $\alpha\pa*{d} = \alpha\pa*{n, d}$ and $\beta\pa*{d} = \beta\pa*{n, d}$ for some $n \ge 2d$.
Using the shorthanded notation, \eqref{eq:alpha-beta-power-series} is immediate result of \autoref{pre:alpha-beta-independence}.

\begin{lem}
\label{pre:alpha-beta-values}
For all $n > 2d$,
\begin{align*}
\alpha\pa*{d} &= \pa*{1-p}\sum_{m=0}^{n-1} \alpha\pa*{m,d} p^m + \alpha\pa*{n, d} p^n \qquad\text{and} \\
\beta\pa*{d} &= \pa*{1-p^{-1}}\sum_{m=0}^{n-1} \beta\pa*{m,d} p^{-m} + \beta\pa*{n, d} p^{-n}.
\end{align*}
\end{lem}
\begin{proof}
By setting $t = p$ in \cite[eq. (38)]{bhargava2021density} and writing $n-1$ instead of $n$, we obtain
\begin{equation*}
\Acl_d\pa*{p} = \pa*{1-p}\sum_{m=0}^{n-1} \alpha\pa*{m,d} p^m + \alpha\pa*{n, d} p^n.
\end{equation*}
Plugging \eqref{eq:abi:power-series-identities} into the last equation gives the equality for $\alpha\pa*{d}$.
And the other equality is obtained by applying the inversion $p\leftrightarrow 1/p$ and \cite[Theorem 1.(a)]{bhargava2021density}.
\end{proof}

Next, we prove two lemmas regrading the values of $\alpha\pa*{n, d}$, $\beta\pa*{n, d}$ that are not covered in \cite{bhargava2021density}.
\begin{lem}
\label{pre:roots-in-haar-polynomial}
Let $f$ be a random polynomial distributed uniformly in $\ZZ_p\br*{X}_n^1$. Then for any $0 < d \le n /2$ we have that
\begin{equation*}
\Ex\br*{\rdset{f_0}} = \beta\pa*{d},
\end{equation*}
where $f_0\pa*{X} = f\pa*{pX}$.
\end{lem}
\begin{proof}
Let $m \le n$ be a non-negative integer and assume that $h \in P_{m,n}$ occurs.
So $h$ is distributed uniformly in $P_{m,n}$ and by \autoref{pre::uniform-distribution-in-P} we get that there exists two random polynomials $f$ and $g$ distributed uniformly in $P_{m,m}$ and $P_{0,n-m}$, respectively, such that $h=fg$.

The map $x \mapsto px$ is a bijection from integer roots of $h_0$ to integer roots of $f$.
Indeed, if $x$ is an integer root of $h_0$ then $f\pa*{px} g\pa*{px} \equiv h_0\pa*{x} = 0$ and since $g\pa*{px} \equiv g\pa*{0} \not\equiv 0 \pmod{p}$ we get that $f\pa*{px} = 0$.
For the other direction, if $y \in \ZZ_p$ is a root of $f$ then $y^m \equiv 0 \pmod{p}$ and then $p \mid y$.
So there exists $x \in\ZZ_p$ such that $y=px$ and $h_0\pa*{x} = 0$ follows immediately.
Therefore, we have that $\rdset{h_0} = \rdset{f}$.

Since $\ZZ_p\br*{X}^1_n = \bigsqcup_{m=0}^n P_{m,n}$, we use the law of total expectation to get
\begin{equation*}
\begin{aligned}
\Ex\br*{\rdset{h_0}}
  &= \sum_{m=0}^n \Ex\br*{\rdset{f_0} \cond f \in P_{m,n}} \Pr\pa*{f \in P_{m,n}} \\
  &= \sum_{m=0}^n \Ex\br*{\rdset{f} \cond f \in P_{m,n}} \Pr\pa*{f \in P_{m,n}}.
\end{aligned}
\end{equation*}
We recall that $f$ is distributed uniformly on $P_{m,m}$ when $h \in P_{m,n}$, hence $\Ex\br*{\rdset{f} \cond h \in P_{m,n}} = \beta\pa*{n,d}$.
Moreover, the probability $\Pr\pa*{f \in P_{m,n}}$ equals $\pa*{p-1}/p^{m+1}$ when $m < n$ and $1/p^n$ when $m=n$.
Therefore,
\begin{equation*}
\begin{aligned}
\Ex\br*{\rdset{h_0}}
  &= \sum_{m=0}^{n-1} \beta\pa*{m,d}\cdot \frac{p-1}{p^{m+1}}  + \beta\pa*{n,d} p^{-n}  \\
  &= \pa*{1 - p^{-1}} \sum_{m=0}^n \beta\pa*{m,d} p^{-m} + \beta\pa*{n,d}p^{-n}.
\end{aligned}
\end{equation*}
And using \autoref{pre:alpha-beta-values} finish the proof.
\end{proof}
\begin{lem}
\label{alpha-beta-asymptotic}
For any integers $0\le d \le n$,
\begin{align*}
  \log_p \alpha\pa*{n, d} &= -\frac{d^2}{2\pa*{p-1}} + O\pa*{d\log d} \qquad\text{and} \\
  \log_p \beta\pa*{n, d} &= -\frac{pd^2}{2\pa*{p-1}} + O\pa*{d\log d}
\end{align*}
as $d \to \infty$.
\end{lem}
\begin{proof}
Let $f$ be a random polynomial distributed uniformly on $\ZZ_p\br*{X}^1_n$, so $\alpha\pa*{n, d} = \Ex\br*{\rdset{f}}$.
We take a look at the values of $\alpha\pa*{n, d}$ when $n=d$.
In this case we have at most $d$ roots, hence $\rdset{f}$ is an indicator function of the event that $f$ has exactly $d$ roots i.e.\ $f$ totally split.
So $\alpha\pa*{d,d} = \Pr\pa*{f\text{ totally splits}}$ and this probability has an asymptotic formula in \cite[Theorem 5.1]{buhler2006probability} which implies
\begin{equation*}
\log_p \alpha\pa*{d, d} = - \frac{d^2}{2\pa*{p-1}} + O\pa*{d\log d}.
\end{equation*}

Setting $n=d$ in \eqref{eq:pre:beta_in_alpha_terms} gives
\begin{equation*}
\beta\pa*{d,d} = p^{-\binom{d}{2}} \alpha\pa*{d, d}.
\end{equation*}
Since $\binom{d}{2} = d^2/2 + O\pa*{d}$ we obtain
\begin{equation*}
\log_p \beta\pa*{d, d} = - \frac{pd^2}{2\pa*{p-1}} + O\pa*{d \log d}.
\end{equation*}

From the definition of big-O notations we that there exists a constant $C_0 > 0$ such that for any $d \ge 0$
\begin{align}
\label{eq:aba:alpha-const-bound}
\abs*{\log_p \alpha\pa*{d,d} + \frac{d^2}{2\pa*{p-1}}} &< C_0 d \log d \qquad\text{and}\\
\label{eq:aba:beta-const-bound}
\abs*{\log_p \beta\pa*{d,d} + \frac{pd^2}{2\pa*{p-1}}} &< C_0 d \log d.
\end{align}

We continue our proof by bounding $\log_p\alpha\pa*{n, d}$ and $\log_p\beta\pa*{n, d}$ from both sides.

\subsubsection*{Upper bound:}
We prove by induction on $n$ the following inequalities
\begin{align}
\label{eq:aba:alpha-upper-bound}
\log_p \alpha\pa*{n,d} &< - \frac{d^2}{2\pa*{p-1}} + C_0 d \log d \qquad\text{and}\\
\label{eq:aba:beta-upper-bound}
\log_p \beta\pa*{n,d} &<  - \frac{pd^2}{2\pa*{p-1}} + C_0 d \log d.
\end{align}
For the base of the induction we take $n = d$.
In this case, \eqref{eq:aba:alpha-upper-bound} and \eqref{eq:aba:beta-upper-bound} are immediate implication of \eqref{eq:aba:alpha-const-bound} and \eqref{eq:aba:beta-const-bound} respectively.

For $n > d$, consider the inner sum, $\sum_{d_0+\dots+d_{p-1} = d} \prod_{i=0}^{p-1} \beta\pa*{n_i, d_i}$, of \eqref{eq:pre:alpha_in_beta_terms} in the case that $\bar f = \pa*{X - r}^n$ for some $r\in \FF_p$.
If there exists $i\ne r$ such that $d_i \ne 0$ then $\beta\pa*{n_i, d_i} = \beta\pa*{0, d_i} = 0$ by \eqref{eq:alpha-beta-initial} and the product eliminates.
Therefore, this sum has only one non-zero summand which obtained when $d_r = d$ and $d_i = 0$ for $i \ne r$.
Hence, when $\bar f = \pa*{X - r}^n$
\begin{equation*}
\sum_{d_0+\dots+d_{p-1} = d} \; \prod_{i=0}^{p-1} \beta\pa*{n_i, d_i} = \beta\pa*{n, d}.
\end{equation*}
We plug this in \eqref{eq:pre:alpha_in_beta_terms} and get
\begin{equation*}
\alpha\pa*{n,d}
  = p^{-n} \sum_{\substack{\bar f \in \FF_p\br*{X}_n^1\\ n_i < n}} \; \sum_{d_0+\dots+d_{p-1} = d} \; \prod_{i=0}^{p-1} \beta\pa*{n_i, d_i} + p^{-n + 1} \beta\pa*{n, d}.
\end{equation*}
We change the order of summation:
\begin{equation*}
\alpha\pa*{n,d}
  = p^{-n} \sum_{d_0+\dots+d_{p-1} = d} \; \sum_{\substack{\bar f \in \FF_p\br*{X}_n^1\\ n_i < n}} \;  \prod_{i=0}^{p-1} \beta\pa*{n_i, d_i} + p^{-n + 1} \beta\pa*{n, d}.
\end{equation*}
We note that if $n_i < d_i$ for some $i$ then $\beta\pa*{n_i, d_i}=0$ and the product eliminates.
Hence,
\begin{equation}
\label{eq:aba:alpha_nd-clean}
\alpha\pa*{n,d}
  = p^{-n} \sum_{d_0+\dots+d_{p-1} = d} \; \sum_{\substack{\bar f \in \FF_p\br*{X}_n^1\\ d_i \le n_i < n}} \;  \prod_{i=0}^{p-1} \beta\pa*{n_i, d_i} + p^{-n + 1} \beta\pa*{n, d}.
\end{equation}

We take a look at the product of \eqref{eq:aba:alpha_nd-clean} for some non-negative integers $d_0,\dots,d_{p-1}$ such that $d_0 + \dots + d_{p-1} = d$ and polynomial $\bar f \in \FF_p\br*{X}_n^1$ such that $d_i \le n_i < n$ for all $i$. Using \eqref{eq:aba:beta-upper-bound} from the induction's assumption we get that
\begin{equation*}
\begin{aligned}
\log_p \pa*{\prod_{i=0}^{p-1} \beta\pa*{n_i, d_i}}
  &< \sum_{i=1}^{p-1}\pa*{-\frac{p d_i^2}{2\pa*{p-1}} + C_0 d_i\log d_i} \\
  &\le -\frac{p}{2\pa*{p-1}} \sum_{i=0}^{p-1} d_i^2 + C_0 d \log d.
\end{aligned}
\end{equation*}
By Cauchy-Schwarz inequality, $d^2 = \pa*{\sum_{i=0}^{p-1} d_i}^2 \le p \sum_{i=0}^{p-1} d_i^2$.
Hence,
\begin{equation}
\label{eq:aba:beta-prod-bound-2}
\log_p \pa*{\prod_{i=0}^{p-1} \beta\pa*{n_i, d_i}}
  < -\frac{d^2}{2\pa*{p-1}} + C_0 d\log d.
\end{equation}
Define the constant $A_d$ by
\begin{equation*}
\log _p A_d = -\frac{d^2}{2\pa*{p-1}} + C_0 d\log d.
\end{equation*}
So we can write \eqref{eq:aba:beta-prod-bound-2} as
\begin{equation*}
\prod_{i=0}^{p-1} \beta\pa*{n_i, d_i} < A_d.
\end{equation*}

We put this in the outer sum of \eqref{eq:aba:alpha_nd-clean} and get
\begin{equation}
\label{eq:aba:alpha-outer-sum-bound}
\begin{aligned}
\sum_{d_0+\dots+d_{p-1} = d} \; \sum_{\substack{\bar f \in \FF_p\br*{X}_n^1\\ d_i \le n_i < n}} \; \prod_{i=0}^{p-1} \beta\pa*{n_i, d_i}
  &< \sum_{d_0+\dots+d_{p-1} = d} \; \sum_{\substack{\bar f \in \FF_p\br*{X}_n^1\\ d_i \le n_i < n}} A_d\\
  &= A_d \sum_{d_0+\dots+d_{p-1} = d} \; \sum_{\substack{\bar f \in \FF_p\br*{X}_n^1\\ d_i \le n_i < n}} 1.
\end{aligned}
\end{equation}
We look at the sums in right most side of \eqref{eq:aba:alpha-outer-sum-bound}, and change the order of summation so
\begin{equation*}
\sum_{d_0+\dots+d_{p-1} = d} \; \sum_{\substack{\bar f \in \FF_p\br*{X}_n^1\\ d_i \le n_i < n}} 1
  = \sum_{\substack{\bar f \in \FF_p\br*{X}_n^1\\ n_i < n}} \; \sum_{\substack{d_0+\dots+d_{p-1} = d \\ d_i \le n_i}} 1.
\end{equation*}
We add to the outer sum the $p$ polynomials of the form $\pa*{X-r}^n$.
For each of those polynomials the inner sum is equal $1$ since all but one $n_i$ is equal $0$.
Therefore,
\begin{equation*}
\sum_{d_0+\dots+d_{p-1} = d} \; \sum_{\substack{\bar f \in \FF_p\br*{X}_n^1\\ d_i \le n_i < n}} 1
  = \sum_{\bar f \in \FF_p\br*{X}_n^1} \; \sum_{\substack{d_0+\dots+d_{p-1} = d \\ d_i \le n_i}} 1 - p.
\end{equation*}
Changing order of summation in the right most side gives
\begin{equation}
\label{eq:aba:composition-counting}
\sum_{d_0+\dots+d_{p-1} = d} \; \sum_{\substack{\bar f \in \FF_p\br*{X}_n^1\\ d_i \le n_i < n}} 1
  = \sum_{d_0+\dots+d_{p-1} = d} \#\set*{\bar f \in \FF_p\br*{X}_n^1 : d_i \le n_i} - p.
\end{equation}
We consider the set $\set*{\bar f \in \FF_p\br*{X}_n^1 : n_r \ge d_r}$.
This set is the set of all polynomials $\bar f \in \FF_p\br*{X}_n^1$ such that $\prod_{r=0}^{p-1} \pa*{X-r}^{d_r} \mid \bar f$.
Hence,
\begin{equation}
\label{eq:aba:polynomial-counting}
\# \set*{\bar f \in \FF_p\br*{X}_n^1 : n_r \ge d_r} = p^{n-d}.
\end{equation}
Plugging \eqref{eq:aba:polynomial-counting} into \eqref{eq:aba:composition-counting} gives
\begin{equation*}
\sum_{d_0+\dots+d_{p-1} = d} \; \sum_{\substack{\bar f \in \FF_p\br*{X}_n^1\\ d_i \le n_i < n}} 1
  = \sum_{d_0+\dots+d_{p-1} = d} p^{n-d} - p
  \le p^n - p,
\end{equation*}
where the inequality is true since the number of summands in the sum is at most $p^d$.
We plug the last inequality\ into \eqref{eq:aba:alpha-outer-sum-bound} to get
\begin{equation*}
\sum_{d_0+\dots+d_{p-1} = d} \; \sum_{\substack{\bar f \in \FF_p\br*{X}_n^1\\ d_i \le n_i < n}} \; \prod_{i=0}^{p-1} \beta\pa*{n_i, d_i}
  < A_d \pa*{p^n - p}.
\end{equation*}
We put the last inequality in \eqref{eq:aba:alpha_nd-clean} and apply simple calculation we get that
\begin{equation}
\label{eq:aba:alpha-beta-ineq}
p^{n-1} \alpha\pa*{n, d} - \beta\pa*{n, d} < \pa*{p^{n-1} - 1} A_d.
\end{equation}

Next, we use \eqref{eq:aba:alpha-upper-bound} from the induction's assumption to infer
\begin{equation*}
\alpha\pa*{s, d} < A_d
\end{equation*}
for $d \le s < n$.
We plug this into \eqref{eq:pre:beta_in_alpha_terms}:
\begin{equation}
  \label{eq:aba:beta_nd}
  \beta\pa*{n,d}
  < p^{-\binom{n}{2}} \alpha\pa*{n, d} + \pa*{p-1} \sum_{d\le s < r < n} p^{-\binom{r+1}{2}} p^s A_d.
\end{equation}
We split the sum in the left most side into two sums and apply some basic computations
\begin{equation*}
  \begin{aligned}
    \sum_{d\le s < r < n} p^{-\binom{r+1}{2}} p^s
    &= \sum_{r = d+1}^{n-1} p^{-\binom{r+1}{2}} \sum_{s=d}^{r-1} p^s \\
    &= \sum_{r = d+1}^{n-1} p^{-\binom{r+1}{2}}\cdot \frac{p^r - p^d}{p-1} \\
    &= \frac{1}{p-1} \sum_{r = d+1}^{n-1} \pa*{p^{-\binom{r}{2}} - p^{-\binom{r+1}{2} + d}} \\
    &\le \frac{1}{p-1} \sum_{r = d+1}^{n-1} \pa*{p^{-\binom{r}{2}} - p^{-\binom{r+1}{2}}} \\
    &= \frac{1}{p-1} \pa*{p^{-\binom{d+1}{2}} - p^{-\binom{n}{2}}}.
  \end{aligned}
\end{equation*}
We combine this with \eqref{eq:aba:beta_nd} to create the inequality
\begin{equation}
  \label{eq:aba:beta-alpha-ineq}
  \beta\pa*{n,d} - p^{-\binom{n}{2}} \alpha\pa*{n,d}
  < \pa*{p^{-\binom{d+1}{2}} - p^{-\binom{n}{2}}} A_d.
\end{equation}

Adding \eqref{eq:aba:alpha-beta-ineq} and \eqref{eq:aba:beta-alpha-ineq} gives
\begin{equation*}
  \pa*{p^{n-1} - p^{-\binom{n}{2}}}\alpha\pa*{n, d}
  < \pa*{p^{n-1} + p^{-\binom{d+1}{2}} - 1 -p^{-\binom{n}{2}}}  A_d.
\end{equation*}
We have that $p^{-\binom{d+1}{2}} - 1 \le 0$ so
\begin{equation}
  \label{eq:aba:final-alpha-bound}
  \alpha\pa*{n, d}
  < \frac{p^{n-1} + p^{-\binom{d+1}{2}} - 1 -p^{-\binom{n}{2}}}{p^{n-1} - p^{-\binom{n}{2}}} \cdot  A_d
  \le  A_d.
\end{equation}
Taking $\log_p$ on both side of the inequality gives \eqref{eq:aba:alpha-upper-bound}.

We use \eqref{eq:aba:final-alpha-bound} and \eqref{eq:aba:beta-alpha-ineq} to get the bound
\begin{equation*}
  \beta\pa*{n,d} < p^{-\binom{n}{2}} \alpha\pa*{n,d} + \pa*{p^{-\binom{d+1}{2}} - p^{-\binom{n}{2}}} A_d \le p^{-\binom{d+1}{2}} A_d.
\end{equation*}
Taking $\log_p$ on both side of the inequality  gives
\begin{equation*}
\log_p \beta\pa*{n, d}
  < - \binom{d+1}{2} - \frac{d^2}{2\pa*{p-1}} + C_0 d \log d.
\end{equation*}
Since $\binom{d+1}{2} > d^2 / 2$, we get that
\begin{equation}
\label{eq:aba:final-beta-bound}
\log_p \beta\pa*{n, d}
  < - \frac{pd^2}{2\pa*{p-1}} + C_0 d \log d.
\end{equation}
The inequality \eqref{eq:aba:final-alpha-bound} after taking $\log_p$ and \eqref{eq:aba:final-beta-bound} finish the induction.

\subsubsection*{Lower Bound:}
We start with bounding $\beta\pa*{n, d}$. For $n = d$, by \eqref{eq:aba:beta-const-bound} we have
\begin{equation}
\label{eq:aba:beta-dd-lower-bound}
\log_p \beta\pa*{d, d} > - \frac{pd^2}{2\pa*{p-1}} - C_0 d\log d.
\end{equation}

For $n = d + 1$, let $g$ be a random polynomial distributed uniformly on $P_{n,n}$, so that $\beta\pa*{n, d} = \Ex\br*{\rdset{g}}$.
The polynomial $g$ cannot miss exactly one root, so $\rdset*{g} = d + 1$ or $\rdset*{g} < d$.
In the later case, we have that $\rdset{g} = 0$.
And when $\rdset*{g} = d + 1$, i.e.\ $g$ totally split, we have that $\rdset{g} = p$.
Also the probability that $g$ totally split is $\beta\pa*{d+1, d+1}$.
Therefore, we use \eqref{eq:aba:beta-const-bound} to get that
\begin{equation*}
\begin{aligned}
\log_p \beta\pa*{d + 1, d} &= \log_p \pa*{p \beta\pa*{d + 1, d + 1}} \\
  &> 1 - \frac{p\pa*{d+1}^2}{2\pa*{p-1}} - C_0 \pa*{d + 1}\log \pa*{d + 1}.
\end{aligned}
\end{equation*}
So there exists a constant $C_1 > 0$ such that
\begin{equation}
\label{eq:aba:beta-dd1-lower-bound}
\log_p \beta\pa*{d + 1, d}
  > - \frac{p d^2}{2\pa*{p-1}} - C_1 d \log d.
\end{equation}

Next, let $n > d + 1$.
We look at \eqref{eq:pre:beta_in_alpha_terms} and omit all terms except the summand where $s=d$ and $r=d+1$:
\begin{equation*}
\beta\pa*{n, d} > \pa*{p-1} p^{-\binom{d + 2}{2}} p^d \alpha\pa*{d, d}.
\end{equation*}
Taking $\log_p$ in both sides and using \eqref{eq:aba:alpha-const-bound} gives
\begin{equation*}
\log_p \beta\pa*{n, d}
  > \log_p \pa*{p-1} - \binom{d + 2}{2} + d - \frac{d^2}{2\pa*{p-1}} - C_0 d \log d.
\end{equation*}
Since $\binom{d+2}{2} = d^2 / 2 + O\pa*{d}$, then
\begin{equation}
\label{eq:aba:beta-nd-lower-bound-restricted}
\log_p \beta\pa*{n, d}
  > - \frac{p d^2}{2\pa*{p-1}} - C_2 d \log d,
\end{equation}
for some constant $C_2 > 0$.
From \eqref{eq:aba:beta-dd-lower-bound}, \eqref{eq:aba:beta-dd1-lower-bound} and \eqref{eq:aba:beta-nd-lower-bound-restricted}, there exists $C_3 > 0$ such that for all $n\ge d$
\begin{equation}
\label{eq:aba:beta-nd-lower-bound-log}
\log_p \beta\pa*{n, d}
  > - \frac{p d^2}{2\pa*{p-1}} - C_3 d \log d.
\end{equation}
We also define the constant $B_d$ by
\begin{equation*}
\log_p B_d = - \frac{p d^2}{2\pa*{p-1}} - C_3 d \log d.
\end{equation*}
So applying exponent with base $p$ on both sides of \eqref{eq:aba:beta-nd-lower-bound-log} gives
\begin{equation}
\label{eq:aba:beta-nd-lower-bound}
\beta\pa*{n, d}
  > B_d.
\end{equation}

Next we bound the values of $\alpha\pa*{n, d}$.
We change the order of summation in \eqref{eq:pre:alpha_in_beta_terms}, so
\begin{equation*}
\alpha\pa*{n, d} = p^{-n} \sum_{d_0 + \dots + d_{p-1} = d} \; \sum_{\bar f \in \FF_p\br*{X}_n^1} \; \prod_{r=0}^{p-1} \beta\pa*{n_r, d_r}
\end{equation*}
As before, if $n_r < d_r$ for some $r$ then $\beta\pa*{n_r, d_r} = 0$ and the product eliminates.
Therefore,
\begin{equation*}
\alpha\pa*{n, d} = p^{-n} \sum_{d_0 + \dots + d_{p-1} = d} \; \sum_{\substack{\bar f \in \FF_p\br*{X}_n^1 \\ n_r \ge d_r}} \; \prod_{r=0}^{p-1} \beta\pa*{n_r, d_r}.
\end{equation*}
Using \eqref{eq:aba:beta-nd-lower-bound} gives
\begin{equation*}
\begin{aligned}
\alpha\pa*{n, d}
  &> p^{-n} \sum_{d_0 + \dots + d_{p-1} = d} \; \sum_{\substack{\bar f \in \FF_p\br*{X}_n^1 \\ n_r \ge d_r}} \; \prod_{r=0}^{p-1} B_{d_r} \\
  &= p^{-n} \sum_{d_0 + \dots + d_{p-1} = d} \#\set*{\bar f \in \FF_p\br*{X}_n^1 : n_r \ge d_r} \prod_{r=0}^{p-1} B_{d_r}.
\end{aligned}
\end{equation*}
We consider the set $\set*{\bar f \in \FF_p\br*{X}_n^1 : n_r \ge d_r}$.
This set is the set of all polynomials $\bar f \in \FF_p\br*{X}_n^1$ such that $\prod_{r=0}^{p-1} \pa*{X-r}^{d_r} \mid \bar f$.
Hence, its size is $p^{n-d}$ and the last equation become
\begin{equation}
\label{eq:aba:alpha-lower-bound-2}
\alpha\pa*{n, d} > p^{-d} \sum_{d_0 + \dots + d_{p-1} = d} \; \prod_{r=0}^{p-1} B_{d_r}
\end{equation}

Note that $\sum_{r=0}^{p-1} \floor*{\pa*{d+r} / p} = d$.
So we omit all terms in the sum of \eqref{eq:aba:alpha-lower-bound-2} except when $d_r = \floor{\pa*{d+r} / p}$ to get
\begin{equation*}
\alpha\pa*{n, d} > p^{-d} \prod_{i=0}^{p-1} B_{\floor*{\pa*{d+i} / p}}.
\end{equation*}
Taking $\log_p$ on both sides gives
\begin{equation*}
\begin{aligned}
\log_p \alpha\pa*{n, d}
  &> -d - \sum_{i=0}^{p-1} \pa*{\frac{p}{2\pa*{p-1}} \floor*{\frac{d+i}{p}}^2 + C_3 \floor*{\frac{d+i}{p}} \log \floor*{\frac{d+i}{p}}} \\
  &\ge -d - \sum_{i=0}^{p-1} \frac{p}{2\pa*{p-1}} \pa*{\frac{d+i}{p}}^2 - C_3 \sum_{i=0}^{p-1}  \pa*{\frac{d+i}{p}} \log \pa*{\frac{d+i}{p}} \\
  &\ge - \frac{p^2}{2\pa*{p-1}} \pa*{\frac{d+p}{p}}^2 - C_4  d \log d.
\end{aligned}
\end{equation*}
for some constants $C_4>0$.
Since $\pa*{\frac{d+p}{p}}^2 = d^2 / p^2 + O\pa*{d}$, there exists a constant $C_5 > 0$ such that
\begin{equation}
\label{eq:aba:alpha-lower-bound-log}
\log_p \alpha\pa*{n, d}
  \ge - \frac{d^2}{2\pa*{p-1}} - C_5  d \log d.
\end{equation}
The inequalities \eqref{eq:aba:alpha-lower-bound-log} and \eqref{eq:aba:beta-nd-lower-bound-log} gives the required lower bounds for the proof.
\end{proof}

We define the values $\gamma\pa*{d}$ using the power series equality \eqref{eq:gamma-def}.
From this definition we have that:
\begin{equation}
\label{eq:gamma-fromula}
\gamma\pa*{d} = \sum_{d_1 + \dots + d_{p-1} = d} \; \prod_{r=1}^{p-1} \beta\pa*{d_r},
\end{equation}
where the sum runs on all non-negative integers $d_1,\dots,d_{p-1}$ such that $d_1 + \dots + d_{p-1} = d$.
Finally, we give an asymptotic estimate for the values of $\gamma\pa*{d}$.

\begin{lem}
\label{gamma-estimate}
We have that
\begin{equation*}
  \log_p\gamma\pa*{d} = -\frac{p d^2}{2\pa*{p-1}^2} + O\pa*{d \log d},
\end{equation*}
when $d \to \infty$.
\end{lem}
\begin{proof}
By \autoref{alpha-beta-asymptotic}, there exists $C_0 > 0$ such that for all $d \ge 0$ we have
\begin{equation}
\label{eq:ge:beta-const-bound}
\abs*{\log_p \beta\pa*{d} + \frac{pd^2}{2\pa*{p-1}}} < C_0 d \log d.
\end{equation}
We continue our proof by bounding $\log_p\gamma\pa*{d}$ from both sides.
\subsubsection*{Upper Bound:}
We look at the product of \eqref{eq:gamma-fromula} for some non-negative integers $d_1,\dots,d_{p-1}$ such that $d_1 + \dots + d_{p-1} = d$, so by \eqref{eq:ge:beta-const-bound} we infer
\begin{equation*}
\begin{aligned}
\log_p \prod_{r=1}^{p-1} \beta\pa*{d_r}
  &< \sum_{r=1}^{p-1} \pa*{- \frac{pd_r^2}{2\pa*{p-1}} + C_0 d_r \log d_r} \\
  &\le -\frac{p}{2\pa*{p-1}} \sum_{r=1}^{p-1} d_r^2 + C_0 \sum_{r=1}^{p-1}d_r \log d.
\end{aligned}
\end{equation*}
By Cauchy-Schwarz inequality, $d^2 = \pa*{\sum_{r=1}^{p-1} d_r}^2 \le \pa*{p-1} \sum_{r=0}^{p-1} d_r^2$.
Hence,
\begin{equation}
\label{eq:ge:log-prod-upper-bound-2}
\log_p \prod_{r=1}^{p-1} \beta\pa*{d_r} < - \frac{p d^2}{2 \pa*{p-1}^2} + C_0 d \log d.
\end{equation}
Define the constant $G_d$ by
\begin{equation*}
\log_p G_d = - \frac{p d^2}{2 \pa*{p-1}^2} + C_0 d \log d.
\end{equation*}

We plug \eqref{eq:ge:log-prod-upper-bound-2} into \eqref{eq:gamma-fromula} to obtain
\begin{equation*}
\gamma\pa*{d}
  < \sum_{d_1+\dots+d_{p-1}=d} G_{d}
  = \binom{d + p-1}{p-1} G_d \le p^d G_d.
\end{equation*}
We take $\log_p$ on both sides of the inequality, so
\begin{equation*}
\begin{aligned}
\log_p \gamma\pa*{d}
  &< d - \frac{p d^2}{2 \pa*{p-1}^2} + C_0 d \log d \\
  &< - \frac{p d^2}{2 \pa*{p-1}^2} + C_1 d \log d,
\end{aligned}
\end{equation*}
for some $C_1 > 0$.
And the last inequality gives the required upper bound.

\subsubsection*{Lower Bound:}
We take a look on \eqref{eq:gamma-fromula}.
Note that $\sum_{r=1}^{p-1} \floor*{\pa*{d+r-1}/\pa*{p-1}} = d$, hence
\begin{equation*}
\gamma\pa*{d} \ge \prod_{r=1}^{p-1} \beta\pa*{\floor*{\frac{d + r - 1}{p-1}}}.
\end{equation*}
Taking $\log_p$ on both sides gives
\begin{equation}
\label{eq:ge:gamma-log-lower-bound}
\log_p \gamma\pa*{d} \ge \sum_{r=1}^{p-1} \log_p \beta\pa*{\floor*{\frac{d + r - 1}{p-1}}}.
\end{equation}
We use \eqref{eq:ge:beta-const-bound} in \eqref{eq:ge:gamma-log-lower-bound} to get
\begin{equation*}
\begin{aligned}
\log_p \gamma\pa*{d}
&\ge - \sum_{r=1}^{p-1} \pa*{\frac{p}{2 \pa*{p-1}} \floor*{\frac{d + r - 1}{p-1}}^2 + C_0 \floor*{\frac{d + r - 1}{p-1}} \log \floor*{\frac{d + r - 1}{p-1}}} \\
&\ge - \sum_{r=1}^{p-1} \pa*{\frac{p}{2 \pa*{p-1}} \pa*{\frac{d + r - 1}{p-1}}^2 + C_0 \pa*{\frac{d + r - 1}{p-1}} \log \pa*{\frac{d + r - 1}{p-1}}} \\
&\ge - \sum_{r=1}^{p-1} \frac{p}{2 \pa*{p-1}} \pa*{\frac{d + p}{p-1}}^2 + C_2 d \log d \\
&\ge - \frac{p}{2 \pa*{p-1}^2}d^2 + C_3 d \log d,
\end{aligned}
\end{equation*}
for some constants $C_2,C_3 > 0$.
And the lower bound is shown as needed.
\end{proof}

\subsection{Random walks}
\label{sec:rnd-walks}

Let $k, m$ be positive integers and let $\rv_{0},\dots,\rv_{n-1}$ be random variables satisfying \autoref{main-assumption}.
Set $V=\pa*{\rmod{p^k}}^m$, for some vectors $\v_0,\v_1,\dots,\v_n$ in $V$, we construct the random walk over the additive group $\pa*{V,+}$ whose $n$-th step is $\sum_{i=0}^n \rv_i \v_i$ (we set $\rv_n = 1$ for convenience).

For two vectors $\u,\w\in V$, we denote by $\ang*{\u,\w}$ the formal dot product i.e.
\[
\ang*{\u,\w} = u_1 w_1 + \dots + u_m w_m.
\]
For a non-zero vector $\u\in V$, we call the number of vectors in $\v_0,\dots,\v_n$ such that $\ang*{\u, \v_i} \ne 0$, the \emph{$\u$-weight of $\v_0,\dots,\v_n$}, and we denote it by $\weight_\u\pa*{\v_0,\dots,\v_n}$.
We define the \emph{minimal weight of $\v_{0},\dots,\v_n$}
to be
\[
\sigma\pa*{\v_0, \dots, \v_n}
  = \min_{\u\in V\setminus\set{\vze}} \weight_\u \pa*{\v_0,\dots,\v_n}.
\]

The relation between $\tau$ from \autoref{main-assumption}, $\sigma$ and the $n$-step of the random walk is:
\begin{prop}
\label{prop:zero-random-walk}
For any $S \subseteq V$. We have:
\[
\Pr\pa*{\sum_{i=0}^n\rv_i \v_i \in S} = \frac{\#S}{\#V} + O\pa*{\#S \exp\pa*{-\frac{\tau\sigma\pa*{\v_0,\dots,\v_n}}{p^{2k}}}},
\]
as $n\to \infty$.
\end{prop}

To prove this proposition we follow the proof of \cite[Proposition 12]{shmueli2021expected}.
The proof of both propositions is mostly the same except for few adjustments allowing $\rv_0, \dots, \rv_{n-1}$ to have different laws from each other.

Let $\mu_0,\dots,\mu_n$ be the laws of $\rv_0,\dots,\rv_n$ and $\nu$ be the law of $\sum_{i=0}^{n}\rv_i \v_i$.
For each index $i$ and positive integer $k$, let $\mu_i^{\pa{k}}$ be the pushforward of $\mu_i$ to $\rmod{p^k}$ and for $k=1$ we also denote $\mu_i'=\mu_i^{\pa{1}}$.
Those measures satisfy the following
\begin{equation*}
  \mu_i^{\pa{k}}\pa*{\bar{x}}=\mu_i\pa*{\bar{x}+p^k\ZZ_p}
  \qquad\text{and}\qquad
  \mu_i'\pa*{\bar{x}}=\mu_i\pa*{\bar{x}+p\ZZ_p}.
\end{equation*}
From \autoref{main-assumption}, we have that $1 - \sum_{\bar{x}\in \rmod{p}} \mu_i'\pa*{\bar{x}}^2 > \tau$ for all $0 < i < n$.

Let $\delta_{\w}$ be the Dirac measure on $V$, i.e.
\begin{equation}
  \label{eq:dirac-delta}
  \delta_\w\pa*{\u} = \begin{cases}
    1, &\u = \w ,\\
    0, &\u \ne \w.
  \end{cases}
\end{equation}
We write $\mu_i.\delta_{\w}$ for the following probability measure on $V$:
\begin{equation}
  \label{eq:small-conv}
  \mu_i.\delta_\w\pa*{\cdot}=\sum_{\bar{x}\in\rmod{p^k}}\mu_i^{\pa{k}}\pa*{\bar{x}}\delta_{x\w}\pa*{\cdot}.
\end{equation}
With this notation, we can write:
\begin{equation}
  \label{eq:random-walk-conv}
  \nu=\mu_0.\delta_{\v_0}\ast\mu_1.\delta_{\v_1}\ast\dots\ast\mu_n.\delta_{\v_n}.
\end{equation}
where $*$ is the convolution operator.

In this section we denote the Fourier transform by $\widehat{\cdot}$
and we let $\zeta$ be a primitive $p^k$-th root of unity.
So for any function $f\colon V \to \CC$ we have the following relations
\begin{align}
  \hat{f}\pa*{\u}&=\sum_{\w\in V}f\pa*{\w}\zeta^{-\ang*{ \u,\w} }\quad\text{and}\label{eq:fourier-trasform} \\ f\pa*{\u}&=\frac{1}{\#V}\sum_{\w\in V}\hat{f}\pa*{\w}\zeta^{\ang*{ \u,\w} }\text{.}\label{eq:inv-fourier-trasform}
\end{align}

The following lemma and its proof are based on \cite[Lemma 13]{shmueli2021expected} and \cite[lemma 31]{breuillard2019irreducibility}.
\begin{lem}
  \label{lem:1-step-fourier-bound}
  Let $\u,\w\in V$ and integer $0 < i < n$.
  If $\ang*{\u,\w} \ne0$ then
  \[
  \abs*{\widehat{\mu_i.\delta_\w}\pa*{\u}} < \exp\pa*{-\frac{\tau}{p^{2k}}}.
  \]
\end{lem}
\begin{proof}
  By direct computation using \eqref{eq:fourier-trasform} and \eqref{eq:small-conv} we get
  \begin{align*}
    \abs*{\widehat{\mu_i.\delta_\w}\pa*{\u}}^2
      &= \sum_{\x,\y\in V}\mu_i.\delta_\w\pa*{\x}\mu_i.\delta_\w\pa*{\y}\zeta^{\ang*{ \x-\y,\u} }\\
      &= \sum_{\x,\y\in V} \sum_{\bar{x},\bar{y}\in\rmod{p^k}} \mu_i^{\pa{k}}\pa*{\bar{x}}\delta_{\bar{x}\w}\pa*{\x} \mu_i^{\pa{k}}\pa*{\bar{y}}\delta_{\bar{y}\w}\pa*{\y} \zeta^{\ang*{\x-\y,\u}}.
  \end{align*}
  Then from \eqref{eq:dirac-delta}
  \begin{equation*}
    \abs*{\widehat{\mu_i.\delta_{\w}}\pa*{\u}}^2
      = \sum_{\bar{x},\bar{y}\in\rmod{p^k}} \mu_i^{\pa{k}}\pa*{\bar{x}} \mu_i^{\pa{k}}\pa*{\bar{y}} \zeta^{\pa*{\bar{x}-\bar{y}}\ang*{\w,\u}}.
  \end{equation*}
  We denote by $\lift{\bar{t}}$ the lift of $\bar{t} \in \rmod{p^k}$
  to the interval $\left(-\frac{p^k}{2},\frac{p^k}{2}\right]\cap\ZZ$.
  Since $\abs*{\widehat{\mu_i.\delta_\w}\pa*{\u}}^2\in\RR$ and
  $\Re\pa*{\zeta^{\bar{t}}}\le1-2\lift{\bar{t}}^2/p^{2k}$, we get
  \begin{align*}
    \abs*{\widehat{\mu_i.\delta_\w}\pa*{\u}}^2
    &= \sum_{\bar{x},\bar{y}\in\rmod{p^k}} \mu_i^{\pa{k}}\pa*{\bar{x}} \mu_i^{\pa{k}}\pa*{\bar{y}} \Re\pa*{\zeta^{\pa*{\bar{x}-\bar{y}}\ang*{\w,\u} }}\\
    & \le\sum_{\bar{x},\bar{y}\in\rmod{p^k}} \mu_i^{\pa{k}}\pa*{\bar{x}} \mu_i^{\pa{k}}\pa*{\bar{y}} \pa*{1-\frac{2\lift{\pa*{\bar{x}-\bar{y}}\ang*{\w,\u}}^2}{p^{2k}}}\\
    & =1-\frac{2}{p^{2k}} \sum_{\bar{x},\bar{y}\in\rmod{p^k}} \mu_i^{\pa{k}}\pa*{\bar{x}} \mu_i^{\pa{k}}\pa*{\bar{y}} \lift{\pa*{\bar{x}-\bar{y}}\ang*{ \w,\u}}^2.
  \end{align*}
  If $p\nmid \bar{x}-\bar{y}$ then $\pa*{\bar{x}-\bar{y}}\ang*{\w,\u}$ is non-zero. So
  \begin{equation*}
    \lift{\pa*{\bar{x}-\bar{y}}\ang*{\w,\u}}^2 \ge 1,
  \end{equation*}
  hence
  \begin{equation}
    \abs*{\widehat{\mu.\delta_{\w}}\pa*{\u}}^{2}
    \le 1 - \frac{2}{q^{2}} \sum_{\substack{
        \bar{x},\bar{y}\in\rmod{p^k} \\
        p\nmid \bar{x}-\bar{y}
      }
    } \mu_i^{\pa{k}}\pa*{\bar{x}} \mu_i^{\pa{k}}\pa*{\bar{y}}
    \label{eq:1-step-fourier-bound-1}\text{.}
  \end{equation}

  Since $\mu_i'$ is also the pushforward measure of $\mu_i^{\pa{k}}$, we have
  \begin{equation*}
    \mu_i' \pa*{\bar{x}'} = \sum_{\substack{
        x\in\rmod{p^k} \\
        \bar{x}'\equiv\bar{x}\pmod{p}}
      } \mu_i^{\pa{k}} \pa*{\bar{x}}.
  \end{equation*}
  Hence
  \begin{equation*}
    \sum_{\substack{
        \bar{x},\bar{y}\in\rmod{p^k} \\
        p\nmid \bar{x}-\bar{y}
      }
    } \mu_i^{\pa{k}}\pa*{\bar{x}} \mu_i^{\pa{k}}\pa*{\bar{y}}
    = \sum_{\substack{
          \bar{x}',\bar{y}'\in\rmod{p} \\
          \bar{x}'\ne \bar{y}'
        }
      } \mu_i'\pa*{\bar{x}'}\mu_i'\pa*{\bar{y}'}.
  \end{equation*}
  By direct computation
  \begin{align*}
    \sum_{\substack{
        \bar{x},\bar{y}\in\rmod{p^k} \\
        p\nmid \bar{x}-\bar{y}
      }
    } \mu_i^{\pa{k}}\pa*{\bar{x}} \mu_i^{\pa{k}}\pa*{\bar{y}}
    &= \sum_{\bar{x}',\bar{y}'\in\rmod{p}} \mu_i'\pa*{\bar{x}'}\mu_i'\pa*{\bar{y}'} - \sum_{\bar{x}'\in\rmod{p}} \mu_i'\pa*{\bar{x}'}^2 \\
    & =\pa*{\sum_{\bar{x}'\in\rmod{p}}\mu_i'\pa*{\bar{x}'}}^2 - \sum_{\bar{x}'\in\rmod{p}} \mu_i'\pa*{\bar{x}'}^2\\
    & =1-\sum_{\bar{x}'\in\rmod{p}} \mu_i'\pa*{\bar{x}'}^2 > \tau.
  \end{align*}
  Plugging this into \eqref{eq:1-step-fourier-bound-1} and using the inequality $1-t\le\exp\pa*{-t}$, we get
  \begin{align*}
    \abs*{\widehat{\mu.\delta_{\w}}\pa*{\u}}^2
    < 1-\frac{2\tau}{q^{2}} \le \exp\pa*{-\frac{2\tau}{q^{2}}}
    \text{.}
  \end{align*}
  We finish the proof by taking square root on both sides of the inequality.
\end{proof}
\begin{lem}
  \label{lem:fourier-bound}
  Let $\u\in V\setminus\set*{\vze} $. Then there exists $C_0>0$ such that
  \begin{equation*}
    \abs*{\hat{\nu}\pa*{\u}} < C_0 \exp\pa*{-\frac{\tau\weight_\u\pa*{\v_0,\dots,\v_n}}{p^{2k}}}.
  \end{equation*}
\end{lem}

\begin{proof}
  We define the following set $I\pa*{\u}=\set*{0 < i < n:\ang*{\u,\v_i} \ne0} $, so that
  \begin{equation*}
    \weight_\u\pa*{\v_1,\dots,\v_{n-1}} = \#I\pa*{\u}
  \end{equation*}
  by definition.
  For $i\in I\pa*{\u}$, \autoref{lem:1-step-fourier-bound} infers that
  \begin{equation}
    \label{eq:bound-step-in-set}
    \abs*{\widehat{\mu_i.\delta_{\v_i}}\pa*{\u}}\le\exp\pa*{-\frac{\tau}{p^{2k}}}.
  \end{equation}

  Otherwise, for $i\notin I\pa*{\u}$ we have that
  \begin{equation}
    \label{eq:bound-step-not-in-set}
    \abs*{\widehat{\mu_i.\delta_{\v_i}}\pa*{\u}}
    \le \sum_{\w\in V}\abs*{\mu_i.\delta_{\v_i}\pa*{\w}\zeta^{-\ang*{\u,\w}}} = 1.
  \end{equation}
  By \eqref{eq:random-walk-conv}, \eqref{eq:bound-step-in-set}, \eqref{eq:bound-step-not-in-set} and since the Fourier transform maps convolutions to products we get
  \begin{equation}
  \label{eq:rw:bound-all}
  \begin{aligned}
    \abs*{\hat{\nu}\pa*{\u}}
    &= \abs*{\prod_{i=0}^{n}\widehat{\mu_i.\delta_{\v_i}}\pa*{\u}}\\
    &\le \abs*{\prod_{i\in I\pa*{\u}}\exp\pa*{-\frac{\tau}{p^{2k}}}} \cdot \abs*{\prod_{i\notin I\pa*{\u}}1} \\
    &= \exp\pa*{-\frac{\tau\#I\pa*{\u}}{p^{2k}}} \\
    &= \exp\pa*{-\frac{\tau\weight_\u\pa*{\v_1,\dots,\v_{n-1}}}{p^{2k}}}.
  \end{aligned}
  \end{equation}

  From the definition of $\u$-weight we get the following inequality
  \begin{equation*}
    \weight_\u\pa*{\v_1,\dots,\v_{n-1}} \ge \weight_\u\pa*{\v_0,\dots,\v_n} - 2.
  \end{equation*}
  Plugging this inequality into \eqref{eq:rw:bound-all} gives
  \begin{equation*}
    \abs*{\hat{\nu}\pa*{\u}} \le \exp\pa*{\frac{2\tau}{p^{2k}}} \exp\pa*{-\frac{\tau\weight_\u\pa*{\v_0,\dots,\v_n}}{p^{2k}}}.
  \end{equation*}
  Since $\tau < 1$, $p \ge 2$ and $k\ge1$ we have that $\exp\pa*{2\tau/p^{2k}} < e$ which finish the proof.
\end{proof}

\begin{proof}[Proof of \autoref{prop:zero-random-walk}]
We have that $\nu$ is the law of $\sum_{i=0}^n \rv_i \v_i$, hence
\begin{equation*}
  \Pr\pa*{\sum_{i=0}^n \rv_i \v_i \in S} = \sum_{\u \in S} \nu\pa*{\u}.
\end{equation*}
Using the triangle inequality gives
\begin{equation}
  \label{eq:rw:prob-triangle-bound}
  \abs*{\Pr\pa*{\sum_{i=0}^n \rv_i \v_i \in S} - \frac{\#S}{\#V}}
    \le \sum_{\u\in S} \abs*{\nu\pa*{\u} - \frac{1}{\#V}}
\end{equation}

Since $\nu$ is a probability measure on $V$, we have by \eqref{eq:fourier-trasform} that
\begin{equation*}
  \hat{\nu}\pa*{\vze} = \sum_{\w \in V} \nu\pa*{\w} = 1.
\end{equation*}
Hence, by \eqref{eq:inv-fourier-trasform}
\begin{equation}
  \label{eq:prop-of-zero}
  \nu\pa*{\u}
  = \frac{1}{\#V} \sum_{\w\in V} \hat{\nu}\pa*{\w}\zeta^{\ang*{\u,\w}}
  = \frac{1}{\#V} + \frac{1}{\#V}\sum_{\w\in V\setminus\set*{\vze}} \hat{\nu}\pa*{\w}\zeta^{\ang*{\u,\w}}.
\end{equation}
Therefore, by the triangle inequality and \autoref{lem:fourier-bound}
\begin{align*}
  \abs*{\nu\pa*{\u}-\frac{1}{\#V}}
  &\le \frac{1}{\#V} \sum_{\w\in V\setminus\set*{\vze}} \abs*{\hat{\nu}\pa*{\w}} \\
  &\le \frac{1}{\#V} \sum_{\w\in V\setminus\set*{\vze}} C_0 \exp\pa*{-\frac{\tau\sigma\pa*{\v_0,\dots,\v_n}}{p^{2k}}}\\
  & < C_0 \exp\pa*{-\frac{\tau\sigma\pa*{\v_0,\dots,\v_n}}{p^{2k}}}.
\end{align*}
We finish the proof by using this bound in \eqref{eq:rw:prob-triangle-bound}.
\end{proof}

We write the following lemma to give a lower bound for the minimal weight.
\begin{lem}
\label{lem:rnd-walks:cycle-minimal-weight}
Let $\v_0,\dots,\v_n$ be vectors such that $\v_0 \bmod{p},\dots,\v_n\bmod{p}$ have a cycle $t < n$ then
\[
\pa*{\frac{n}{t} - 1} \sigma\pa*{\v_0 \bmod{p},\dots,\v_{t-1}\bmod{p}} \le \sigma\pa*{\v_0,\dots,\v_n}.
\]
Also, if $\v_0 \bmod{p},\dots,\v_n\bmod{p}$ contains a basis for $\FF_p^m$, then
\[
\frac{n}{t} - 1 \le \sigma\pa*{\v_0,\dots,\v_n}.
\]
\end{lem}
\begin{proof}
We have that
\[
\weight_\u \pa*{\v_0,\dots, \v_n} \ge \weight_{\u\bmod p} \pa*{\v_0\bmod{p},\dots, \v_n\bmod{p}},
\]
for all non-zero $\u \in V$.
So it suffices to prove the lemma for $k=1$ and rest will follow.

The first part is clear since,
\[
\weight_\u \pa*{\v_0,\dots,\v_n}
  \ge \sum_{i=0}^{\floor*{n/t} - 1} \weight_\u \pa*{\v_{it},\dots,\v_{it + t - 1}}
  > \pa*{\frac{n}{t} - 1} \weight_\u \pa*{\v_0,\dots,\v_{t-1}}.
\]

For the second part, let $i_1, \dots, i_m$ such that $\v_{i_1}, \dots, \v_{i_m}$ is basis of $\FF_p^m$.
We can assume without loss of generality that $i_j < t$ for all $j=1,\dots,m$.
Hence
\[
\sigma \pa*{\v_0, \dots, \v_{t-1}} \ge \sigma \pa*{\v_{i_1}, \dots, \v_{i_m}},
\]
and it suffices to show that
\[
\sigma\pa*{\v_{i_1}, \dots, \v_{i_m}} \ge 1.
\]
Assume otherwise, so there is a non-zero vector $\u$ such that $\ang*{\u, \v_{i_j}} = 0$ for all $j$.
Hence, $\u$ is a non-trivial solution of the following linear equation:
\begin{equation*}
\pa*{
\begin{matrix}
- & \v_{i_1} & - \\
- & \v_{i_2} & - \\
 & \vdots &  \\
- & \v_{i_m} & -
\end{matrix}
} \u \equiv \vze \pmod{p}.
\end{equation*}
But the matrix is invertible since its rows form a basis of $\FF_p^m$.
So we got a contradiction and the proof is finished.
\end{proof}
\section{A space of polynomials modulo \texorpdfstring{$p^k$}{p\string^k}}
\label{sec:ups-space}
Define the following subset of $\rmod{p^k}\br*{X}$
\begin{equation*}
  \Upsilon_k
    = \set*{\sum_{i=0}^{k-1} \bar{a}_i p^i X^i \cond \forall i, \bar{a}_i \in \rmod{p^{k-i}}}.
\end{equation*}
We have a natural bijection $\rmod{p^k} \times \dots \times \rmod{p} \to \Upsilon_k$, so
\begin{equation}
  \label{eq:ups:size}
  \#\Upsilon_k = \prod_{i=0}^{k-1} p^{k-i} = p^{k\pa*{k+1}/2}.
\end{equation}

\begin{prop}
\label{prop:ups:d-sets-in-ups}
For any positive integer $n$, let $d = d\pa*{n}$ and $k = k\pa*{n}$ be positive integers such that $\limsup_{n\to\infty} d\log n/k^2 < \log p / 8$.
Assume $g\in\ZZ_p\br*{X}_n$ is a random polynomial such that $g\bmod{p^k}$ is distributed uniformly in $\Upsilon_k$.
Then
\[
\Ex\br*{\rdset{g}} = \beta\pa*{d} + O\pa*{p^{\pa*{-\frac{1}{2} + \frac{1}{2}\he_p\pa*{\frac{2d}{k}}}k}}
\]
as $n\to\infty$.
\end{prop}

For $\bar{x}\in\rmod{p^k}$ we say that $\bar{x}$ is \emph{simple root of $g$ modulo $p^k$} if $g\pa*{\bar{x}}\equiv 0\pmod{p^k}$ and $g'\pa*{\bar{x}}\not\equiv0\pmod{p^k}$.
And, we say $\bar{x}\in\rmod{p^k}$ is \emph{non-simple root of $g$ modulo $p^k$} if $g\pa*{\bar{x}}\equiv g'\pa*{\bar{x}}\equiv 0\pmod{p^{k}}$.

\begin{lem}
\label{lem:ups:simple-prob}
For any $k>0$, let $g\in\ZZ_p\br*{X}$
be a random polynomial such that $g\bmod{p^{k}}$ is distributed uniformly
in $\Upsilon_k$. Then we have that
\begin{equation*}
\Pr\pa*{g\text{ has a non-simple root modulo }p^k} = O\pa*{p^{-k}}
\end{equation*}
as $k\to\infty$.
\end{lem}
\begin{proof}
Let $\mathcal{M}$ be the set of all non-simple roots of $g$ modulo $p^k$.
We write $g\pa*{X} = \rvs_0 + \rvs_1 p X + \rvs_2 p^2 X^2 + \dots$ and then for a fixed $\bar{x}\in\rmod{p^k}$ we have
\begin{equation}
\label{eq:ups:single-non-simple-prob}
\begin{aligned}
\Pr\pa*{\bar{x}\in\mathcal{M}}
  &= \Pr\pa*{\eqsys{g\pa*{\bar{x}} \equiv 0 \pmod{p^k}}{g'\pa*{\bar{x}} \equiv 0 \pmod{p^k}}} \\
  &= \Pr\pa*{\eqsys{\rvs_0 \equiv -\pa*{\rvs_1 p \bar{x} + \dots} \pmod{p^k}}{\rvs_1 p \equiv -\pa*{2\rvs_2 p^2 \bar{x}^2 + \dots} \pmod{p^k}}} \\
  &= \Pr\pa*{\eqsys{\rvs_0 \equiv -\pa*{\rvs_1 p \bar{x} + \dots} \pmod{p^k}}{\rvs_1 \equiv -\pa*{2\rvs_2 p \bar{x}^2 + \dots} \pmod{p^{k-1}}}} \\
  &= p^{-2k+1}
  .
\end{aligned}
\end{equation}
The last equality holds true because $\pa*{\rvs_0 \bmod p^k,\rvs_{1} \bmod p^{k-1}}$ is distributed uniformly in $\rmod{p^k} \times \rmod{p^{k-1}}$.

We finish the proof by using union bound and plugging \eqref{eq:ups:single-non-simple-prob}
\begin{align*}
\Pr\pa*{g\text{ has a non-simple root modulo }p^k}
  &= \Pr\pa*{\bigcup_{\bar{x}\in\rmod{p^k}} \set{\bar{x}\in\mathcal{M}}} \\
  &\le \sum_{\bar{x}\in\rmod{p^k}} \Pr\pa*{\bar{x}\in\mathcal{M}} = p^{-k+1}
  . \qedhere
\end{align*}
\end{proof}

\begin{defn}
For a polynomial $g\in\ZZ_p\br*{X}$ we say that $\bar{x}\in\rmod{p^k}$ is a $k$\emph{-Henselian root of $g$} if $g'\pa*{\bar{x}}\not\equiv0\pmod{p^k}$ and there is a lift $\bar{y}$ of $\bar{x}$ in $\rmod{p^{2k-1}}$ such that $g\pa*{\bar{y}}\equiv0\pmod{p^{2k-1}}$.
\end{defn}

We denote by $\hdset{k}{g}$ the number of $d$-sets of $k$-Henselians roots of $g$.
By Newton-Raphson method (\autoref{pre:newton-raphson}) every $k$-Henselian root can be lifted uniquely to a root of $g$ in $\ZZ_p$, hence
\begin{equation*}
\rdset{g} \ge \hdset{k}{g}.
\end{equation*}
Equality holds as follows
\begin{lem}
\label{ups:simple-roots-henselians}
If a polynomial $g\in\ZZ_p\br*{X}$ has only simple roots modulo $p^k$, then $\rdset{g} = \hdset{k}{g}$.
\end{lem}
\begin{proof}
It suffices to show that $\rdset*{g} = \hdset[1]{k}{g}$.
So, we consider the map $x \mapsto x\bmod{p^k}$ and prove that this map is a bijection from integer roots of $g$ to $k$-Henselians roots of $g$.
Indeed, if $x$ is a root of $g$  then $x\bmod{p^k}$ is root of $g$ modulo $p^k$.
Since $g$ has only simple roots modulo $p^k$, we obtain $g'\pa*{\bar{x}} \not\equiv 0 \pmod{p^k}$.
Also, we have that $x\bmod{p^{2k-1}}$ is a lift of $x\bmod{p^k}$ and $g\pa*{x\bmod{p^{2k-1}}} \equiv 0 \pmod{p^{2k-1}}$.
Finally, Newton-Raphson method (\autoref{pre:newton-raphson}) gives that the map is invertible.
\end{proof}

\begin{lem}
\label{lem:ups:henselian-approx}
For any positive integer $n$, let $d = d\pa*{n}$ and $k = k\pa*{n}$ be positive integers such that $\limsup_{n\to\infty} d\log n/k^2 < \log p / 2$.
Then
\[
\Ex\br*{\rdset{g}} = \Ex\br*{\hdset{k}{g}} + O\pa*{p^{\pa*{-1+\he_p\pa*{\frac{d}{k}}}k}}
\]
as $n\to\infty$.
\end{lem}
\begin{proof}
We define the following events:
\begin{enumerate}
  \item $S$ be the event that $g$ has only simple roots modulo $p^k$,
  \item $M$ be the event that $g$ has a non-simple root modulo $p^k$ and $g\not\equiv 0 \pmod{p^k}$ and
  \item $Z$ be the event that $g \equiv 0\pmod{p^k}$.
\end{enumerate}
Clearly, the events are disjoints and $\Pr\pa*{S\sqcup M \sqcup Z} = 1$.
From the total law of expectation we get that
\begin{equation}
\label{eq:ups:total-law-ex}
\Ex\br*{\rdset{g}}
  = \Ex\br*{\rdset{g} \cond S} \Pr\pa*{S}
    + \Ex\br*{\rdset{g} \cond M} \Pr\pa*{M}
    + \Ex\br*{\rdset{g} \cond Z} \Pr\pa*{Z}.
\end{equation}

We know that when $S$ occurs then $\rdset{g} = \hdset{k}{g}$ by \autoref{ups:simple-roots-henselians}, so
\begin{equation*}
\Ex\br*{\rdset{g} \cond S} \Pr\pa*{S}
  = \Ex\br*{\hdset{k}{g} \cond S} \Pr\pa*{S}
  \le \Ex\br*{\hdset{k}{g}}.
\end{equation*}
Recall that $\rdset{g} \ge \hdset{k}{d}$, and with \eqref{eq:ups:total-law-ex} we obtain
\begin{equation}
\label{eq:ups:rdsets-with-big-error-term}
\Ex\br*{\rdset{g}}
  = \Ex\br*{\hdset{k}{g}}
    + O\pa*{\Ex\br*{\rdset{g} \cond M} \Pr\pa*{M}
    + \Ex\br*{\rdset{g} \cond Z} \Pr\pa*{Z}}.
\end{equation}

To bound the error term we first consider $\Ex\br*{\rdset{g} \cond M} \Pr\pa*{M}$.
If $g\not\equiv 0 \pmod{p^k}$ then by \cite[Proposition 5]{shmueli2021expected} we get that $\rdset*{g} < k$.
Hence,
\[
\Ex\br*{\rdset{g} \cond M} = \Ex\br*{\binom{\rdset*{g}}{d} \cond M} \le \binom{k}{d}.
\]
Using binomial coefficient's approximation \cite[Example 11.1.3]{cover2006elements} we get
\begin{equation*}
\Ex\br*{\rdset{g} \cond M} = O\pa*{p^{\he_p\pa*{\frac{d}{k}} k}}.
\end{equation*}
Moreover, from \autoref{lem:ups:simple-prob} we have $\Pr\pa*{M} = O\pa*{p^{-k}}$, so
\begin{equation}
\label{eq:ups:final-non-simple-bound}
\Ex\br*{\rdset{g} \cond M} \Pr\pa*{M}
  = O\pa*{p^{\he_p\pa*{\frac{d}{k}} k} p^{-k}}
  = O\pa*{p^{\pa*{-1 + \he_p\pa*{\frac{d}{k}}} k}}.
\end{equation}

Next we bound $\Ex\br*{\rdset{g} \cond Z} \Pr\pa*{Z}$.
Since $\rdset*{g} \le n$,
\[
\Ex\br*{\rdset{g} \cond Z} = \Ex\br*{\binom{\rdset*{g}}{d} \cond Z} \le \binom{n}{d} \le n^d.
\]
From \eqref{eq:ups:size} we have that $\Pr\pa*{Z} = p^{-k\pa*{k+1}/2} = O\pa*{p^{-k^2/2}}$ so
\begin{equation*}
\Ex\br*{\rdset{g} \cond Z} \Pr\pa*{Z} = O\pa*{n^d p^{-k^2 / 2}} = O\pa*{\exp\pa*{d \log n - \frac{1}{2} k^2 \log p}}.
\end{equation*}
Since $\limsup_{n\to\infty} d\log n / k^2 < \log p / 2$ there exists a constant $c > 0$ such that $d \log n - k^2 \log p / 2 < -c k^2$, hence
\begin{equation}
\label{eq:ups:final-zero-modulo-bound}
\Ex\br*{\rdset{g} \cond Z} \Pr\pa*{Z} = O\pa*{\exp\pa*{-c k^2}}.
\end{equation}
Finally, plugging \eqref{eq:ups:final-non-simple-bound} and \eqref{eq:ups:final-zero-modulo-bound} into \eqref{eq:ups:rdsets-with-big-error-term} finish the proof.
\end{proof}

\begin{proof}[Proof of \autoref{prop:ups:d-sets-in-ups}]
Let $h \in \ZZ_p\br*{X}$ be a random polynomial distributed uniformly in $\ZZ_p\br*{X}^1_k$.
We write $h\pa*{X} = \rvs_0 + \rvs_1 X + \dots + \rvs_{k-1} X^{k-1} + X^k$, and set
\[
h_0\pa*{X} = h\pa*{pX} = \rvs_0 + \rvs_1 p X + \dots + \rvs_{k-1} p^{k-1} X^{k-1} + p^k X^k.
\]
The random variable $\rvs_i$ is distributed according to the normalized Haar measure on $\ZZ_p$ hence $\rvs_i \bmod p^{k-i}$ is distributed uniformly in $\rmod{p^{k-i}}$.

Therefore, $h_0 \bmod{p^k}$ and $g \bmod{p^k}$ has the same distribution which is the uniform distribution on $\Upsilon_k$.
Also $h_0 \bmod{p^{\floor{k/2}}}$ and $g \bmod{p^{\floor{k/2}}}$ are both uniform in $\Upsilon_{\floor*{k/2}}$, so \autoref{lem:ups:henselian-approx} gives
\begin{align}
\label{eq:us:rdset-estimate}
\Ex\br*{\rdset{g}}
  &= \Ex\br*{\hdset{\floor{k/2}}{g}} + O\pa*{p^{\pa*{-1 + \he_p\pa*{\frac{2d}{k}}}\frac{k}{2}}}, \\
\label{eq:us:hdset-estimate}
\Ex\br*{\rdset{h_0}}
  &= \Ex\br*{\hdset{\floor{k/2}}{h_0}} + O\pa*{p^{\pa*{-1 + \he_p\pa*{\frac{2d}{k}}}\frac{k}{2}}}.
\end{align}
Also, $\Ex\br*{\hdset{\floor{k/2}}{g}}$ depends only on the distribution of $g\bmod{p^k}$, thus
\begin{equation*}
\Ex\br*{\hdset{\floor{k/2}}{g}} = \Ex\br*{\hdset{\floor{k/2}}{h_0}}
\end{equation*}
and by subtracting \eqref{eq:us:hdset-estimate} from \eqref{eq:us:rdset-estimate} we get
\begin{equation*}
\Ex\br*{\rdset{g}}
  = \Ex\br*{\rdset{h_0}} + O\pa*{p^{\pa*{-\frac{1}{2} + \frac{1}{2}\he_p\pa*{\frac{2d}{k}}}k}}.
\end{equation*}
Using \autoref{pre:roots-in-haar-polynomial} in the last equation finish the proof.
\end{proof}


\section{The distribution of Hasse derivatives}
 \label{sec:hasse}

In this section we explore properties of the distribution of the Hasse derivatives of $f$, $D^{\pa{j}}f$, modulo powers of $p$.
We recall that the Hasse derivatives are given by the following
\begin{equation}
\label{eq:hasse-derivative}
D^{\pa{j}} f\pa*{r} = \frac{1}{j!} f^{\pa{j}}\pa*{r} = \sum_{i=1}^n \rv_i \binom{i}{j} r^{i-j}.
\end{equation}
Also note that if $\rv_0,\dots,\rv_{n-1}, r \in \ZZ_p$ then also $D^{\pa{j}} f\pa*{r}\in \ZZ_p$, so it is possible to talk about its reduction modulo $p$.

\begin{prop}
\label{prop:hasse:coeff-prob}
Let $f$ be random polynomial with coefficients satisfying \autoref{main-assumption} and let $m = m\pa*{n} <n / p^2$ be a positive integer.
Assume $k_0, \dots, k_{m-1}$ are non-negative integers and $a^{\pa{j}}_r\in\ZZ_{p}$ is $p$-adic integer for all $j\in\set*{0,\dots, m-1}$ and $r\in\set*{1,\dots,p-1}$.
Then
\begin{equation*}
\Pr\pa*{\bigwedge_{\tindex}\quad D^{\pa{j}}f\pa*{r} \equiv a^{\pa{j}}_r\pmod{p^{k_j}}}=p^{-N\pa*{p-1}}+O\pa*{\exp\pa*{-\frac{\tau n}{p^{2k+2}m}+Kmp\log p}}
\end{equation*}
as $n\to\infty$ where $K=\max_{0\le j<m}k_j$, $N=\sum_{j=0}^{m-1}k_j$ and $\tau$ is taken from \autoref{main-assumption}.
\end{prop}

Set $V=\pa*{\rmod{p^K}}^{m\pa*{p-1}}$ and construct a random walk over the additive group $\pa*{V,+}$ as in \autoref{sec:rnd-walks} with the vectors $\v_0,\dots,\v_n$ defined by
\begin{equation}
\label{eq:hasse:step-vectors}
\v_i = \pa*{\binom{i}{j} r^{i-j}}_{\substack{0\le j < m \\ 1 \le r < p}}.
\end{equation}
From \eqref{eq:hasse-derivative} the $n$-step of the random walk is
\begin{equation}
\label{eq:hasse:hasse-vectors}
\sum_{i=0}^{n}\rv_i \v_i = \pa*{D^{\pa{j}}f\pa*{r}}_{\substack{0\le j < m \\ 1 \le r < p}}.
\end{equation}

\begin{lem}
\label{cor:hasse:v-minimal-weight}
Let $\v_0, \dots, \v_{n}$ be vectors as in \eqref{eq:hasse:step-vectors}.
Then
\[
\frac{n}{p^2 m} - 1 < \sigma\pa*{\v_0,\dots,\v_{n}}.
\]
\end{lem}
\begin{proof}
Let $\ell$ be the integer such that $p^{\ell - 1} < m \le p^\ell$.
First, we show that the vectors $\v_0, \dots, \v_n$ has cycle of $\pa*{p-1}p^\ell$ modulo $p$.
By Lucas's Theorem (see \cite{fine1947binomial}) for any $j < m \le p^\ell$
\[
\binom{i + p^\ell}{j} \equiv \binom{i}{j} \pmod{p}.
\]
Also, by Fermat's Little Theorem, $r^{p-1} \equiv 1 \pmod{p}$.
So
\[
\v_{i + \pa*{p-1}p^\ell}
  \equiv \pa*{\binom{i + \pa*{p-1}p^\ell}{j} r^{i + \pa*{p-1}p^\ell - j}}_{\substack{0\le j < m \\ 1 \le r < p}}
  \equiv \pa*{\binom{i}{j} r^{i - j}}_{\substack{0\le j < m \\ 1 \le r < p}}
  \equiv \v_i \pmod{p}.
\]

Next, we show that $\v_0 \bmod{p}, \dots, \v_{n - 1} \bmod{p}$ contains a basis of $\FF_p^{m\pa*{p-1}}$.
We take a look on two vectors sequences $\u_0, \dots, \u_{m-1} \in \FF_p^m$ and
$\w_1, \dots, \w_{p-1} \in \FF_p^{p-1}$ defined by
\[
\u_i = \pa*{\binom{i}{0}, \dots, \binom{i}{m-1}} \quad\text{, and}\quad
\w_i = \pa*{1, 2^{i-1}, \dots, \pa*{p-1}^{i-1}}.
\]
The sequence $\u_0,\dots,\u_{m-1}$ is a basis of $\FF_p^m$ since its vectors form a triangular matrix
\[
\pa*{\begin{matrix}
- & \u_{0} & -\\
- & \u_{1} & -\\
 & \vdots\\
- & \u_{m-1} & -
\end{matrix}}=\pa*{\begin{matrix}
1 & 0 & 0 & \cdots & 0\\
1 & 1 & 0 & \cdots & 0\\
1 & 2 & 1 & \cdots & 0\\
\vdots & \vdots & \vdots & \ddots & \vdots\\
1 & m-1 & \binom{m-1}{2} & \cdots & 1
\end{matrix}}.
\]
The sequence $\w_1, \dots, \w_{p-1}$ is a basis of $\FF_p^{p-1}$ since its vectors are columns of Vandermonde matrix.
Since both sequences are basis of their vector spaces then the set $\set*{\u_{t} \otimes \w_{s} : 0\le t < m, 1 \le s < p}$ is basis of the tensor product $\FF_p^m \otimes \FF_p^{p-1} \cong \FF_p^{m\pa*{p-1}}$.

Define $T:\FF_p^{m\pa*{p-1}} \to \FF_p^{m\pa*{p-1}}$ to be the linear map which multiplies the element in the $\pa*{j,r}$ coordinates by $r^{-j}$ i.e.
\[
T\pa*{\pa*{x^{\pa{j}}_r}_\tindex} = \pa*{r^{-j} x^{\pa{j}}_r}_\tindex.
\]
The map $T$ is an invertible, so the set $B = \set*{T\pa*{\u_{t} \otimes \w_{s}} : 0\le t < m, 1 \le s < p}$ is also a basis of $\FF_p^{m\pa*{p-1}}$.

Moreover, the basis $B$ is contained in $\v_0 \bmod{p}, \dots, \v_{n - 1} \bmod{p}$.
Indeed, for every $0\le t < m$ and $1 \le s < p$ we set $i^{\pa{t}}_s$ to be the positive integer such that $i^{\pa t}_s \equiv t \pmod{p^{\ell}}$ and $i^{\pa t}_s \equiv s - 1 \pmod{p-1}$.
By the Chinese Reminder Theorem there exists such $i^{\pa t}_s$ which also satisfy $i^{\pa t}_s < p^\ell \pa*{p-1} < mp^2 \le n$.
So for any $T\pa*{\u_t \otimes \w_s}\in B$
\begin{align*}
T\pa*{\u_t \otimes \w_s}
  &= T\pa*{\pa*{\binom{t}{j}}_{0\le j<m} \otimes \pa*{r^{s-1}}_{1\le r<p}} \\
  &= T\pa*{\pa*{\binom{t}{j} r^{s-1}}_\tindex} \\
  &= T\pa*{\pa*{\binom{i^{\pa t}_s}{j} r^{i^{\pa t}_s}}_\tindex} \\
  &= \pa*{\binom{i^{\pa t}_s}{j} r^{i^{\pa t}_s-j}}_\tindex \\
  &= \v_{i^{\pa t}_s}.
\end{align*}
We proved the requirements of \autoref{lem:rnd-walks:cycle-minimal-weight}, so
\[
\sigma\pa*{\v_0,\dots,\v_{n-1}} \ge \frac{n}{\pa*{p-1}p^{\ell}} - 1 > \frac{n}{p^2 m} - 1. \qedhere
\]
\end{proof}

\begin{proof}[Proof of \autoref{prop:hasse:coeff-prob}]
Define
\[
S=\set*{\pa*{x^{\pa{j}}_r}_{\substack{0\le j < m \\ 1 \le r < p}} \in V \cond \forall j, r, \quad x^{\pa{j}}_r\equiv a^{\pa{j}}_r\pmod{p^{k_j}}} \subseteq V,
\]
so by \eqref{eq:hasse:hasse-vectors} we get
\[
\Pr\pa*{\bigwedge_\tindex D^{\pa{i}}f\pa*{r} \equiv a^{\pa{i}}_r\pmod{p^{k_i}}}\\
  = \Pr\pa*{\sum_{i=0}^{n} \rv_i \v_i \in S}.
\]
\autoref{prop:zero-random-walk} gives
\begin{equation*}
\Pr\pa*{\sum_{i=0}^{n} \rv_i \v_i \in S}
  =\frac{\#S}{\#V}+O\pa*{\#S\exp\pa*{-\frac{\tau\sigma\pa*{\v_0,\dots,\v_n}}{p^{2K}}}}.
\end{equation*}
Since $\#S=\prod_{i=0}^{m-1}p^{\pa*{K-k_i}\pa*{p-1}}=p^{Km\pa*{p-1}} p^{-N\pa*{p-1}}\le p^{Kmp}$, $\#V=p^{Km\pa*{p-1}}$,
and by \autoref{cor:hasse:v-minimal-weight}
\[
\Pr\pa*{\sum_{i=0}^{n}\rv_{i}\v_{i}\in S}
  = p^{-N\pa*{p-1}}+O\pa*{\exp\pa*{-\frac{\tau n}{p^{2K+2}m}+Kmp\log p}}. \qedhere
\]
\end{proof}
\section{Proofs of the main theorems}
\label{sec:main}

For a polynomial $f \in \ZZ_p\br*{X}$ and $r \in \set*{0,\dots, p-1}$, we define $f_r\pa*{X}$ to be the polynomial $f\pa*{r + pX}$.
Note that by Taylor expansion we have that
\begin{equation}
\label{eq:main:fr-taylor}
f_r\pa*{X} = D^{\pa{0}}f\pa*{r} + D^{\pa{1}}f\pa*{r} p X + D^{\pa{2}} f\pa*{r} p^2 X^2 + \dots \quad.
\end{equation}
Each integer root of $f_r$ correspond to an integer root of $f$ which congruent to $r$ modulo $p$.
Hence, we divide the roots of $f$ into $p$ sets by their values modulo $p$ and get
\begin{equation}
\label{eq:dsets-split}
\rdset{f} = \sum_{d_0 + \dots + d_{p-1} = d} \; \prod_{r=0}^{p-1} \rdset[d_r]{f_r}.
\end{equation}
Denote the sequence of $p-1$ polynomials $\pa*{f_1, f_2, \dots, f_{p-1}}$ with $\polysplit$. We stress that $f_0$ does not appear in this sequence.

\begin{lem}
\label{lem:main:almost-uniform}
Let $f$ be a random polynomial with coefficients satisfying \autoref{main-assumption} and let $0<\varepsilon<1$.
Then there exists $C_0>0$ such that for any positive integer $k \le \frac{\left(1-\varepsilon\right)\log n}{2\log p}$ and a sequence of polynomials $\hseq = \pa*{h_1, \dots, h_{p-1}}\in\Upsilon_{k}^{p-1}$, we have
\[
\Pr\pa*{\polysplit \equiv \hseq \pmod{p^k}}
  = \pa*{\frac{1}{\#\Upsilon_k}}^{p-1}+O\pa*{\exp\pa*{-C_0n^\varepsilon}}
\]
as $n\to\infty$.
Moreover, $C_0$ is dependent only on $\varepsilon$ and $\tau$ from \autoref{main-assumption}.
\end{lem}
\begin{proof}
As each $h_r \in \Upsilon_k$, it is of the form $h_r\pa*{X} = \bar{a}^{\pa{0}}_r + \bar{a}^{\pa{1}}_r p X + \dots +\bar{a}^{\pa{k-1}}_r p^{k-1} X^{k-1}$.
So by \eqref{eq:main:fr-taylor} we have
\[
f_r\equiv h_r \pmod{p^k} \iff
  D^{\pa{i}}f\pa*{r} \equiv \bar{a}^{\pa{j}}_r  \pmod{p^{k-j}},
  \quad j=0,\dots,k-1.
\]
We apply \autoref{prop:hasse:coeff-prob} with $m=k$, $k_i = k - i$ and $a^{\pa{i}}_r$ be some lift of $\bar{a}^{\pa{i}}_r$, so $K= k$, $N = k \pa*{k+1} / 2$ and
\begin{align*}
\Pr\pa*{\polysplit \equiv \hseq \pmod{p^k}}
  &= \Pr\pa*{\bigwedge_{r=1}^{p-1}  f_r\equiv h_r\pmod{p^k}} \\
  &= \Pr\pa*{\bigwedge_{\tindexarg{k}} D^{\pa{j}} f\pa*{r}\equiv \bar{a}_r^{\pa{i}}\pmod{p^{k-j}}} \\
  &= p^{-k\pa*{k+1}\pa*{p-1}/2}+O\pa*{\exp\pa*{-\frac{\tau n}{p^{2k+2}k}+k^2 p\log p}}
\end{align*}
where $\tau$ is taken from \autoref{main-assumption}.
This finish the proof.
Indeed, the main term is $\pa*{1/\#\Upsilon_k}^{p-1}$ by \eqref{eq:ups:size}
and by the assumption on $k$, the error term is $O\pa*{\exp\pa*{-C_0n^\varepsilon}}$ as needed.
\end{proof}

\begin{lem}
\label{ex-sum-prod}
Let $f$ be a random polynomial with coefficients satisfying \autoref{main-assumption}.
Then for any $d = d\pa*{n}$ such that $\limsup_{n\to\infty} d / \log n < \pa*{16 \log p}^{-1}$ and for any $\varepsilon > 0$  there exists a constant $C > 0$, as defined in \eqref{eq:main-exponent-const}, satisfying
\begin{equation*}
  \Ex\br*{\sum_{d_1 + \dots + d_{p-1} = d} \; \prod_{r=1}^{p-1} \rdset[d_r]{f_r}} =
    \gamma\pa*{d} + O\pa*{n^{-C+\varepsilon}}
\end{equation*}
as $n\to\infty$.
\end{lem}

\begin{proof}
Set
\begin{equation}
\label{eq:rdsetu}
\rdsetu{f} = \sum_{d_1 + \dots + d_{p-1} = d} \; \prod_{r=1}^{p-1} \rdset[d_r]{f_r}.
\end{equation}
Let $k=\floor{\frac{\pa*{1-\dlt}\log n}{2\log p}}$ where $\dlt$ is a positive real to be defined later.
So  we apply the law of total expectation and
\autoref{lem:main:almost-uniform} to get $C_0>0$ such that
\begin{equation}
\label{eq:main:total-expectation}
\begin{aligned}
\Ex\br*{\rdsetu{f}}
  &= \sum_{\hseq\in\Upsilon_k^{p-1}}\Ex\br*{\rdsetu{f}\cond \polysplit \equiv \hseq\pmod{p^k}} \Pr\pa*{\polysplit \equiv \hseq\pmod{p^k}} \\
  &= \sum_{\hseq\in\Upsilon_k^{p-1}}\Ex\br*{\rdsetu{f}\cond \polysplit \equiv \hseq\pmod{p^k}} \pa*{\pa*{\frac{1}{\#\Upsilon_k}}^{p-1}+O\pa*{\exp\pa*{-C_0 n^\dlt}}}.
\end{aligned}
\end{equation}
Since $\rdsetu{f} \le \rdset{f} \le n^d$ (see \eqref{eq:dsets-split} and \eqref{eq:rdsetu}) and $\#\Upsilon_k=p^{k\pa*{k+1}/2}$ (see \eqref{eq:ups:size}) then we can bound the error term in \eqref{eq:main:total-expectation}
\begin{equation*}
\begin{aligned}
& \hspace{-6em}
\sum_{\hseq\in\Upsilon_k^{p-1}}\Ex\br*{\rdsetu{f}\cond \polysplit \equiv \hseq \pmod{p^k}} \exp\pa*{-C_0 n^\dlt} \\
  &=O\pa*{n^d p^{k\pa*{k+1}\pa*{p-1}/2}\exp\pa*{-C_0 n^\dlt}} \\
  &=O\pa*{\exp\pa*{d\log n + k^2 \log p - C_0 n^\dlt}} \\
  &=O\pa*{\exp\pa*{C_1 \log^2 n - C_0 n^\dlt}} \\
  &=O\pa*{\exp\pa*{-C_2 n^\dlt}},
\end{aligned}
\end{equation*}
for some constants $C_1, C_2 > 0$.
So \eqref{eq:main:total-expectation} becomes
\begin{equation}
\label{eq:main:total-expectation-clean}
\Ex\br*{\rdsetu{f}} = \pa*{\frac{1}{\#\Upsilon_k}}^{p-1} \sum_{\hseq\in\Upsilon_k^{p-1}}\Ex\br*{\rdsetu{f}\cond \polysplit \equiv \hseq\pmod{p^k}} + O\pa*{\exp\pa*{-C_2 n^{\dlt}}}.
\end{equation}
To evaluate the main term, we use linearity of expectation and change the order of summation to get
\begin{multline}
  \label{eq:main:conditend-dsets}
  \pa*{\frac{1}{\#\Upsilon_k}}^{p-1} \sum_{\hseq\in\Upsilon_k^{p-1}}\Ex\br*{\rdsetu{f}\cond \polysplit \equiv \hseq \pmod{p^k}} \\
  \begin{aligned}
    &= \pa*{\frac{1}{\#\Upsilon_k}}^{p-1} \sum_{\hseq\in\Upsilon_k^{p-1}} \; \sum_{d_1 + \dots + d_{p-1} = d} \Ex\br*{ \prod_{r=1}^{p-1} \rdset[d_r]{f_r} \cond \polysplit \equiv \hseq \pmod{p^k}} \\
    &= \sum_{d_1 + \dots + d_{p-1} = d}\pa*{\frac{1}{\#\Upsilon_k}}^{p-1} \sum_{\hseq\in\Upsilon_k^{p-1}}\Ex\br*{ \prod_{r=1}^{p-1} \rdset[d_r]{f_r} \cond \polysplit \equiv \hseq \pmod{p^k}}
  \end{aligned}
\end{multline}

Next we evaluate the summands of the outer sum.
Let $\gseq = \pa*{g_1, \dots, g_{p-1}}$ be a sequence of random polynomials in $\ZZ_p\br*{X}_n$ distributed according to the law
\[
\Pr\pa*{\gseq \in E}
  = \pa*{\frac{1}{\#\Upsilon_k}}^{p-1}\sum_{\hseq\in\Upsilon_k^{p-1}}\Pr\pa*{\polysplit \in E \cond \polysplit \equiv \hseq \pmod{p^k}}
\text{,}\quad E\subseteq \pa*{\ZZ_p\br*{X}_n}^{p-1}.
\]
This distribution is well-defined for $n$ sufficiently large, since $\Pr\pa*{\polysplit \equiv \hseq \pmod{p^k}}$ is bounded away from zero by \autoref{lem:main:almost-uniform}.

On the one hand, we have
\begin{equation}
\label{eq:main:expected-prod-1}
\Ex\br*{\prod_{r=1}^{p-1}\rdset[d_r]{g_r}}
  = \pa*{\frac{1}{\#\Upsilon_k}}^{p-1} \sum_{\hseq\in\Upsilon_k^{p-1}} \Ex\br*{\prod_{r=1}^{p-1}\rdset[d_r]{f_r}\cond \polysplit \equiv \hseq \pmod{p^k}}.
\end{equation}
On the other hand, the sequence $\gseq \bmod{p^k}$ is distributed uniformly in $\pa*{\Upsilon_k}^{p-1}$, hence $g_1\bmod{p^k},\dots,g_{p-1}\bmod{p^k}$ are i.i.d.\ and distributed uniformly in $\Upsilon_k$.
Set $L = \limsup_{n\to\infty} d / \log n$ so for any $\dlt < 1 - \sqrt[3]{16L\log p}$
\begin{equation*}
\limsup_{n\to\infty} \frac{d\log n}{k^2}
  \le \frac{2 L \log^2 p}{\pa*{1-\dlt}^3} < \frac{\log p}{8}.
\end{equation*}
Hence, we can apply \autoref{prop:ups:d-sets-in-ups} and obtain
\begin{equation}
\label{eq:main:expected-prod-2}
\Ex\br*{\prod_{r=1}^{p-1} \rdset[d_r]{g_r}}
  = \prod_{r=1}^{p-1} \Ex\br*{\rdset[d_r]{g_r}}
  = \prod_{r=1}^{p-1} \pa*{\beta\pa*{d_r} + O\pa*{p^{\pa*{-\frac{1}{2} + \frac{1}{2}\he_p\pa*{\frac{2d}{k}}}k}}}.
\end{equation}

By \autoref{alpha-beta-asymptotic} we have that $\beta\pa*{d_r} = O\pa*{1}$ so expanding the left most side of \eqref{eq:main:expected-prod-2} gives
\begin{equation}
\label{eq:main:expected-prod-3}
\Ex\br*{\prod_{r=1}^{p-1} \rdset[d_r]{g_r}}
  = \prod_{r=1}^{p-1} \beta\pa*{d_r} + O\pa*{p^{\pa*{-\frac{1}{2} + \frac{1}{2}\he_p\pa*{\frac{2d}{k}}}k}}.
\end{equation}
We get expression for the summands of the outer sum of \eqref{eq:main:conditend-dsets} by combining \eqref{eq:main:expected-prod-1} and \eqref{eq:main:expected-prod-3}:
\begin{equation*}
\pa*{\frac{1}{\#\Upsilon_k}}^{p-1} \sum_{\hseq\in\Upsilon_k^{p-1}} \Ex\br*{\prod_{r=1}^{p-1}\rdset[d_r]{f_r}\cond \polysplit \equiv \hseq \pmod{p^k}}
  = \prod_{r=1}^{p-1} \beta\pa*{d_r} + O\pa*{p^{\pa*{-\frac{1}{2} + \frac{1}{2}\he_p\pa*{\frac{2d}{k}}}k}}.
\end{equation*}

We return to \eqref{eq:main:conditend-dsets} and set the summands of the outer sum:
\begin{multline}
\label{eq:main:conditend-dsets-2}
\pa*{\frac{1}{\#\Upsilon_k}}^{p-1} \sum_{\hseq\in\Upsilon_k^{p-1}}\Ex\br*{\rdsetu{f}\cond \polysplit \equiv \hseq \pmod{p^k}} \\
\begin{aligned}
  &= \sum_{d_1 + \dots + d_{p-1} = d} \pa*{\prod_{r=1}^{p-1} \beta\pa*{d_r} + O\pa*{p^{\pa*{-\frac{1}{2} + \frac{1}{2}\he_p\pa*{\frac{2d}{k}}}k}}} \\
  &= \sum_{d_1 + \dots + d_{p-1} = d} \prod_{r=1}^{p-1} \beta\pa*{d_r} + \binom{p + d - 2}{p-1} O\pa*{p^{\pa*{-\frac{1}{2} + \frac{1}{2}\he_p\pa*{\frac{2d}{k}}}k}} \\
  &= \gamma\pa*{d} + \binom{p + d - 2}{p-1} O\pa*{p^{\pa*{-\frac{1}{2} + \frac{1}{2}\he_p\pa*{\frac{2d}{k}}}k}}.
\end{aligned}
\end{multline}

Consider the error term, so for $n$ sufficiently large we have
\begin{equation}
\label{eq:main:final-error-term-1}
\begin{aligned}
\binom{p + d - 2}{p-1} p^{\pa*{-\frac{1}{2} + \frac{1}{2}\he_p\pa*{\frac{2d}{k}}}k}
  &= O\pa*{d^p p^{\pa*{-\frac{1}{2} + \frac{1}{2}\he_p\pa*{\frac{2d}{k}}}k}} \\
  &= O\pa*{\exp\pa*{p \log d + \pa*{-\frac{1}{2} + \frac{1}{2}\he_p\pa*{\frac{2d}{k}}}k \log p}} \\
  &= O\pa*{\exp\pa*{p \log \pa*{L\log n} + \pa*{1 - \dlt}\pa*{-\frac{1}{4} + \frac{1}{4}\he_p\pa*{\frac{2d}{k}}} \log n}} \\
  &= O\pa*{\exp\pa*{\pa*{1 - \dlt}\pa*{-\frac{1}{4} + \frac{1}{4}\he_p\pa*{\frac{2d}{k}} + \dlt} \log n}}.
\end{aligned}
\end{equation}
Note that since that $L < \pa*{16\log p}^{-1}$ we have that $\limsup_{n\to\infty} 2d/k < 4 L \log p < 1/4$.
Moreover, since $\he_p$ is increasing in $\pa*{0, 1/2}$, for $n$ sufficiently large
\begin{equation*}
 - \frac{1}{4} + \frac{1}{4}\he_p\pa*{\frac{2d}{k}}
  \le \frac{1}{4} - \frac{1}{4}\he_p\pa*{4L\log p} + \dlt = -C + \dlt,
\end{equation*}
where $C$ is defined as in \eqref{eq:main-exponent-const}.
We return to \eqref{eq:main:final-error-term-1} and choose $\dlt < \min\pa*{\varepsilon / \pa*{2+C}, 1- \sqrt[3]{4L\log p}}$ to get that
\begin{equation*}
\binom{p + d - 2}{p-1} p^{\pa*{-\frac{1}{2} + \frac{1}{2}\he_p\pa*{\frac{d}{k}}}k}
  = O\pa*{\exp\pa*{\pa*{1-\dlt}\pa*{-C + 2\dlt}}\log n} = O\pa*{n^{-C + \varepsilon}}.
\end{equation*}
And \eqref{eq:main:conditend-dsets-2} becomes
\begin{equation*}
\pa*{\frac{1}{\#\Upsilon_k}}^{p-1} \sum_{\hseq\in\Upsilon_k^{p-1}}\Ex\br*{\rdsetu{f}\cond \polysplit \equiv \hseq \pmod{p^k}} = \gamma\pa*{d} + O\pa*{n^{-C + \varepsilon}}.
\end{equation*}
Plugging the last equation into \eqref{eq:main:total-expectation-clean} finish the proof.
\end{proof}

\begin{proof}[Proof of \autoref{thm:main}]
When conditioning on $p \nmid \rv_0$ no root is divided by $p$ and then $\rdset*{f_0} = 0$.
Hence $\rdset[0]{f_0} = 1$ and $\rdset[d_0]{f_0} = 0$ for $d_0 \ge 1$.
Plugging that in \eqref{eq:dsets-split} and adding expectation gives
\begin{equation}
\label{eq:dsets-split-main}
\Ex\br*{\rdset{f} \cond p \nmid \rv_0} = \Ex\br*{\sum_{d_1 + \dots + d_{p-1} = d} \prod_{r=1}^{p-1} \rdset[d_r]{f_r} \cond p \nmid \rv_0}.
\end{equation}
Note that \autoref{main-assumption} still holds after conditioning on $p\mid \rv_0$, so we can \autoref{ex-sum-prod} to finish the proof.
\end{proof}

\begin{proof}[Proof of \autoref{main-with-nonunit-root}]
By \autoref{pre:nonunit-roots} we have that almost surely
\begin{equation*}
\rdset[d_0]{f_0} = \begin{cases}
1\qquad&, d_0 = 0, \\
\indic{\rv_0 = 0} &, d_0 = 1, \\
0 &, d_0 > 1.
\end{cases}
\end{equation*}
Using \eqref{eq:dsets-split} gives
\begin{equation*}
\Ex\br*{\rdset{f}} = \Ex\br*{\sum_{d_1 + \dots + d_{p-1} = d} \prod_{r=1}^{p-1} \rdset[d_r]{f_r}} + \Ex\br*{\indic{\rv_0 = 0}\sum_{d_1 + \dots + d_{p-1} = d-1} \prod_{r=1}^{p-1} \rdset[d_r]{f_r}}.
\end{equation*}
By total law of expectation
\begin{equation}
\label{eq:dsets-split-nonunit}
\begin{aligned}
\Ex\br*{\rdset{f}} &= \Ex\br*{\sum_{d_1 + \dots + d_{p-1} = d} \prod_{r=1}^{p-1} \rdset[d_r]{f_r}} \\ &\qquad
+ \Ex\br*{\sum_{d_1 + \dots + d_{p-1} = d-1} \prod_{r=1}^{p-1} \rdset[d_r]{f_r} \cond \rv_0 = 0}\Pr\pa*{\rv_0 = 0}.
\end{aligned}
\end{equation}
And the proof is finished by using \autoref{ex-sum-prod} on each expectation on the right side of \eqref{eq:dsets-split-nonunit}.
Note that this can be done since \autoref{main-assumption} still holds after conditioning on $\rv_0$.
\end{proof}

\section{Large number of roots}
\label{sec:large-root-count}

In this section, we consider the probability that $f$ has a large number of roots relatively to $n$.

\begin{proof}[Proof of \autoref{large-root-count-bound}]
Let $r = r\pa*{n}$ be positive integer and set $d = \floor*{\frac{\log n}{17\log p}}$.
We define the following function $\varphi: \RR_+ \to \RR_+$ by
\begin{equation}
\varphi\pa*{x} = \begin{cases}
0  &, x \le d, \\
\frac{1}{d!} x \pa*{x - 1} \cdots \pa*{x - d + 1} &,x > d.
\end{cases}
\end{equation}
Note that $\varphi$ is non-decreasing function hence Markov inequality gives
\begin{equation}
\label{eq:lrc:markov-ineq}
\Pr\pa*{\rdset*{f} \ge r \cond p \nmid \rv_0} \le \frac{\Ex\br*{\varphi\pa*{\rdset*{f}} \cond p \nmid \rv_0}}{\varphi\pa*{r}},
\end{equation}
when $r > d$.

Assume that
\begin{equation}
\label{eq:lrc:limsup-assumption}
\limsup_{n\to \infty} \frac{d}{r} < 1,
\end{equation}
then for $n$ sufficiently large we have that $r > d$ and \eqref{eq:lrc:markov-ineq} holds.
For any non-negative integer $k$, we have that $\varphi\pa*{k} = \binom{k}{d}$.
In particular, we have that $\varphi\pa*{\rdset*{f}} = \rdset{f}$.
Therefore, \eqref{eq:lrc:markov-ineq} becomes
\begin{equation*}
\Pr\pa*{\rdset*{f} \ge r \cond p \nmid \rv_0} \le \frac{\Ex\br*{\rdset{f} \cond p \nmid \rv_0}}{\varphi\pa*{r}}.
\end{equation*}

From \autoref{thm:main}, we have that $\Ex\br*{\rdset{f} \cond p \nmid \rv_0} = \gamma\pa*{d} + O\pa*{n^{-C_0}}$ for some $C_0 > 0$.
Also from \autoref{gamma-estimate} we have that
\begin{equation*}
\log_p \gamma\pa*{d}
  \le - C_1 d^2 \\
  \le - C_2 \log^2 n.
\end{equation*}
Therefore,
\begin{equation}
\label{eq:lrc:ex-bound}
\Ex\br*{\rdset{f} \cond p \nmid \rv_0} = O\pa*{n^{-C_0}}.
\end{equation}

Next we bound the term $1/\varphi\pa*{r}$ from \eqref{eq:lrc:markov-ineq}.
We have that for $r > d$
\begin{equation*}
\varphi\pa*{r} = \frac{r}{d} \cdot \frac{r - 1}{d - 1} \cdots \frac{r - d + 1}{1} \ge \pa*{\frac{r}{d}}^d.
\end{equation*}
Hence,
\begin{equation}
\label{eq:lrc:phi-bound}
\frac{1}{\varphi\pa*{r}} = O\pa*{\pa*{\frac{d}{r}}^{d}}.
\end{equation}

We plug \eqref{eq:lrc:ex-bound} and \eqref{eq:lrc:phi-bound} into \eqref{eq:lrc:markov-ineq} to get
\begin{equation*}
\begin{aligned}
\Pr\pa*{\rdset*{f} \ge r \cond p \nmid \rv_0}
  &= O\pa*{\pa*{\frac{d}{r}}^{d} \cdot n^{-C_0}} \\
  &= O\pa*{\exp\pa*{d\log \frac{d}{r} - C_0 \log n}} \\
  &= O\pa*{\exp\pa*{\frac{\log n}{17 \log p} \log \frac{d}{r}- C_0 \log n}}.
\end{aligned}
\end{equation*}
From \eqref{eq:lrc:limsup-assumption}, for $n$ sufficiently large $\log\pa*{r/d} < 0$, hence
\begin{equation}
\label{eq:lrc:almost-final}
\Pr\pa*{\rdset*{f} \ge r \cond p \nmid \rv_0} = O\pa*{\exp\pa*{C_3 \log n \log \frac{d}{r}}}
\end{equation}
for some $C_3 > 0$.

We consider the case where $r = \log n$.
So since $p \ge 2$ we have
\begin{equation*}
\limsup_{n\to \infty} \frac{d}{r} = \frac{1}{17 \log p} < \frac{1}{10} < 1.
\end{equation*}
Hence \eqref{eq:lrc:limsup-assumption} holds, and from \eqref{eq:lrc:almost-final} we get
\begin{equation*}
  \Pr\pa*{\rdset*{f} \ge \log n \cond p \nmid \rv_0} = O\pa*{\exp\pa*{- C_3 \log 10 \log n}}.
\end{equation*}
Choosing $K = C_3 \log 10$ finish the proof of \hyperref[lrcb:log-case]{\autoref*{large-root-count-bound}.\eqref{lrcb:log-case}}.

Next, consider the case where $r = n^\lambda$ for some $0 < \lambda \le 1$.
Since $\limsup_{n\to \infty} \frac{d}{r} = 0$, the assumption \eqref{eq:lrc:limsup-assumption} holds.
Thus \eqref{eq:lrc:almost-final} also holds.
Moreover, because
\begin{equation*}
\log \frac{d}{r} \le \log \frac{\log n} {17 \log p} - \lambda \log n,
\end{equation*}
there exists $K > 0$ such that
\begin{equation*}
  \Pr\pa*{\rdset*{f} \ge n^{\lambda} \cond p \nmid \rv_0} = O\pa*{\exp\pa*{- K \log^2 n}}.
\end{equation*}
And the last equation finish the proof of \hyperref[lrcb:log-case]{\autoref*{large-root-count-bound}.\eqref{lrcb:power-case}}.
\end{proof}

\bibliographystyle{bibstyle}
\bibliography{bib/main}

\end{document}